\pdfoutput=1
\documentclass[11pt,reqno,a4paper]{amsart}

\usepackage[utf8]{inputenc}
 
\usepackage[
url=false,
doi=false,
isbn=false,
style=alphabetic,
sorting=anyt,
backend=biber,
maxnames=12,
maxalphanames=6,
giveninits=true]{biblatex}
\AtEveryBibitem{\clearlist{language}}
\AtEveryBibitem{\clearfield{pages}} 
\AtEveryBibitem{\clearfield{eprintclass}}
\AtEveryBibitem{\clearfield{eprinttype}}
\AtEveryBibitem{\clearfield{month}}
\renewbibmacro{in:}{}
\DeclareFieldFormat{eprint}{Preprint: {\tt #1}}
\DeclareFieldFormat[article,inproceedings,inbook]{citetitle}{#1}
\DeclareFieldFormat[article,inproceedings,inbook]{title}{#1} 
\usepackage{setspace} 
\AtBeginBibliography{
  \setstretch{1}
}

\bibliography{main}

\usepackage{amsmath}
\numberwithin{equation}{section}
\usepackage{amsthm}
\usepackage{thmtools}
\usepackage{amssymb}
\usepackage{amsfonts}
\usepackage{mathtools}
\usepackage{hhline}
\usepackage{stmaryrd}
\usepackage{enumitem}
\usepackage{centernot}
\usepackage{mathrsfs}
\usepackage[version=4]{mhchem}
\usepackage{dsfont}
\usepackage{comment}
\usepackage{listings}
\usepackage{mathdots}
\usepackage{nicematrix} 
\usepackage{mathalpha} 
\usepackage{xspace}
\usepackage[dvipsnames]{xcolor}
\usepackage{booktabs}
\usepackage{multirow}

\usepackage[font=footnotesize,labelfont=bf]{caption}

\usepackage{array}
\newcolumntype{L}[1]{>{\raggedright\let\newline\\\arraybackslash\hspace{0pt}}m{#1}}
\newcolumntype{C}[1]{>{\centering\let\newline\\\arraybackslash\hspace{0pt}}m{#1}}
\newcolumntype{R}[1]{>{\raggedleft\let\newline\\\arraybackslash\hspace{0pt}}m{#1}}

\usepackage{tikz-cd}
\usepackage{tikz}
\usetikzlibrary{arrows}
\usetikzlibrary{decorations.pathreplacing}
\usetikzlibrary{matrix}
\usetikzlibrary{calc}
\usetikzlibrary{shapes}
\usetikzlibrary{patterns}
\usetikzlibrary{fit,backgrounds,scopes,plotmarks}

\tikzset{ 
    ncbar angle/.initial=90,
    ncbar/.style={
        to path=(\tikztostart)
        -- ($(\tikztostart)!#1!\pgfkeysvalueof{/tikz/ncbar angle}:(\tikztotarget)$)
        -- ($(\tikztotarget)!($(\tikztostart)!#1!\pgfkeysvalueof{/tikz/ncbar angle}:(\tikztotarget)$)!\pgfkeysvalueof{/tikz/ncbar angle}:(\tikztostart)$)
        -- (\tikztotarget)
    },
    ncbar/.default=0.5cm,
}
\tikzset{square left brace/.style={ncbar=1.5mm}}
\tikzset{square right brace/.style={ncbar=-1.5mm}}
\tikzset{round left paren/.style={ncbar=0.5cm,out=120,in=-120}}
\tikzset{round right paren/.style={ncbar=0.5cm,out=60,in=-60}}

\newtheoremstyle{theoremstyle}
{10pt}      
{5pt}       
{\itshape}  
{}          
{\bfseries} 
{.}         
{ }      
{}          

\newtheoremstyle{algorithmstyle}
{10pt}      
{5pt}       
{}  
{}          
{\bfseries} 
{.}         
{ }      
{}          

\newtheoremstyle{examplestyle}
{10pt}      
{5pt}       
{}          
{}          
{\bfseries} 
{.}         
{ }      
{}          

\makeatletter 
\newcommand{\subalign}[1]{%
  \vcenter{%
    \Let@ \restore@math@cr \default@tag
    \baselineskip\fontdimen10 \scriptfont\tw@
    \advance\baselineskip\fontdimen12 \scriptfont\tw@
    \lineskip\thr@@\fontdimen8 \scriptfont\thr@@
    \lineskiplimit\lineskip
    \ialign{\hfil$\m@th\scriptstyle##$&$\m@th\scriptstyle{}##$\hfil\crcr
      #1\crcr
    }%
  }%
}
\makeatother

\theoremstyle{theoremstyle}
\newtheorem{theorem}{Theorem}[section]
\newtheorem{lemma}[theorem]{Lemma} 

\newtheorem{corollary}[theorem]{Corollary}

\newtheorem{theoremalphabetic}{Theorem}

\theoremstyle{examplestyle}
\newtheorem{example}[theorem]{Example}
\AtBeginEnvironment{example}{%
  \pushQED{\qed}%
}
\AtEndEnvironment{example}{\popQED\endexample}

\newtheorem{definition}[theorem]{Definition}

\newtheorem{remark}[theorem]{Remark}

\usepackage[linesnumbered,ruled,noend]{algorithm2e}

\usepackage[unicode,psdextra,pdfusetitle,bookmarks, bookmarksdepth=2, colorlinks=true, linkcolor=blue!50!black, citecolor=black!70, urlcolor=blue!50!black]{hyperref}
\usepackage[capitalise,noabbrev,nameinlink]{cleveref}
\Crefname{algocf}{Algorithm}{Algorithms}

\definecolor{darkorange}{rgb}{1.0, 0.55, 0.0}
\definecolor{darkblue}{rgb}{0.0, 0.0, 0.55}
\definecolor{darkgreen}{rgb}{0.0, 0.2, 0.13}
\definecolor{darkred}{rgb}{0.75, 0.0, 0.0}

\newcommand{\CC}{\mathbb{C}}
\newcommand{\RR}{\mathbb{R}}
\newcommand{\QQ}{\mathbb{Q}}
\newcommand{\ZZ}{\mathbb{Z}}
\newcommand{\FF}{\mathbb{F}}

\newcommand{\mC}{\mathcal{C}}
\newcommand{\suchthat}{\;\ifnum\currentgrouptype=16 \middle\fi|\;}

\newcommand{\lin}{{\mathrm{lin}}}

\newcommand{\bin}{{\mathrm{bin}}}

\newcommand{\CCt}{\mathbb{C}\{\!\{t\}\!\}}
\newcommand{\RRt}{\mathbb{R}\{\!\{t\}\!\}}

\DeclareMathOperator{\rank}{rank}
\DeclareMathOperator{\grc}{grc}
\DeclareMathOperator{\mrc}{mrc}
\DeclareMathOperator{\rowspan}{rowspan}

\DeclareMathOperator{\id}{id}
\DeclareMathOperator{\im}{im}

\DeclareMathOperator{\conv}{conv}
\DeclareMathOperator{\cone}{cone}
\DeclareMathOperator{\diag}{diag}

\DeclareMathOperator{\mult}{mult}

\DeclareMathOperator{\MV}{MV}
\DeclareMathOperator{\vol}{vol}

\DeclareMathOperator{\relint}{relint}
\DeclareMathOperator{\Span}{Span}

\DeclareMathOperator{\Trop}{Trop}
\DeclareMathOperator{\val}{val}

\DeclareMathOperator{\supp}{supp}
\DeclareMathOperator{\sgn}{sign}

\newcommand{\overbar}[1]{\mkern 2.5mu\overline{\mkern-2.5mu#1\mkern-1.5mu}\mkern 1.5mu}

\DeclareTextFontCommand{\bfemph}{\bfseries\em}
\newcommand{\term}{\bfemph}

\newcommand\restr[2]{{\left.\kern-\nulldelimiterspace #1 \right|_{#2}}}

\makeatletter
\newcommand{\oset}[3][0ex]{%
  \mathrel{\mathop{#3}\limits^{
    \vbox to#1{\kern-2\ex@
    \hbox{$\scriptstyle#2$}\vss}}}}
\makeatother

\makeatletter
\newcommand{\uset}[3][0ex]{%
  \mathrel{\mathop{#3}\limits_{
    \vbox to#1{\kern-7\ex@
    \hbox{$\scriptstyle#2$}\vss}}}}
\makeatother
\makeatletter
\newcommand{\customlabel}[2]{%
   \protected@write \@auxout {}{\string \newlabel {#1}{{#2}{\thepage}{#2}{#1}{}} }%
   \hypertarget{#1}{#2}%
}
\makeatother

\usepackage{xcolor}

\definecolor{juliabg}{RGB}{245,245,245}
\definecolor{jlkeyword}{RGB}{0,0,180}
\definecolor{jlcomment}{RGB}{0,128,0}
\definecolor{jlstring}{RGB}{163,21,21}
\definecolor{jlnumber}{RGB}{150,0,0}
\definecolor{jloperator}{RGB}{0,0,0}
\definecolor{jlfunction}{RGB}{120,0,120}

\lstdefinelanguage{Julia}{
  morekeywords={
    using,function,end,if,else,elseif,for,while,return,import,export,
    struct,mutable,where,begin,let,quote,try,catch,finally,macro,do
  },
  sensitive=true,
  morecomment=[l]\#,
  morestring=[b]",
}

\usepackage[indent,skip=.2\baselineskip]{parskip}

\usepackage{fullpage}



\usepackage[perpage]{footmisc}

\begin{document}

\title{
Root bounds of vertical systems\\ using tropical~geometry
}

\author[Elisenda Feliu, Paul Alexander Helminck, Oskar Henriksson, Yue Ren, Benjamin Schröter and Máté L.\, Telek]{Elisenda Feliu, Paul Alexander Helminck, Oskar Henriksson,\\ Yue Ren, Benjamin Schröter and Máté L.~Telek}

\date{May 8, 2026}
 
\begin{abstract}
Sparse polynomial systems with vertical coefficient dependencies arise naturally when describing the critical points of optimization problems and, when augmented with linear forms, the steady states of chemical reaction networks.
Moreover, any polynomial system is the specialization of such a parametrized system.
We prove that the generic number of complex zeros of an augmented vertically parametrized system is the tropical intersection number of a tropical linear space and a classical linear space.
In the special case when the matroid of the tropical linear space is cotransversal, we express this number as a mixed volume.
We also obtain bounds on the maximal number of positive zeros, which is often the significant number in applications. We derive lower bounds from the number of intersections between positive tropicalizations, and when the positive zeros have  toric structure, we provide upper bounds that are simpler and in some cases smaller than the generic root count. 
The resulting algorithms are implemented in Julia. 
\end{abstract}

\maketitle

\section{Introduction}
\noindent A fundamental problem in computational algebraic geometry is to determine the maximal number of (positive, real or complex) zeros that a sparse polynomial system may have, as the coefficients vary over some given parameter space. When the coefficients are assumed to vary independently of each other, this is a classical problem; over the complex numbers, the answer is given by Bernstein's theorem \cite{Bernstein1975}, and over the positive real numbers, several upper and lower bounds are known \cite{Sottile-book,BihanDickensteinGiaroli2020,BihanDickensteinForsgaard2021}.

Sparse systems that arise in applications often have linear dependencies among the coefficients, and for such systems, the problem of finding root bounds is less explored. Examples of previous work include root bounds for synchronization systems via Newton--Okounkov bodies \cite{BBPMZ23} or Lagrangian systems via toric geometry \cite{BreidingSottileWoodcock22,LindbergMoninRose23}.

The systems of interest in this article are square \emph{vertically parametrized systems}, augmented with linear forms, which take the form
\begin{equation}
\label{eq:augmented_vertically_parametrized_intro}
    F=\big( C x^M,\: Lx-b\big)\in\CC[a,b][x^\pm]^n
\end{equation}
for variables $x=(x_1,\ldots,x_n)$ and  parameters $(a,b)=(a_1,\ldots,a_m,b_1,\ldots,b_d)$. The system is defined by three matrices: a parametric coefficient matrix $C\in\CC[a]^{s\times r}$, where each entry is a linear form in the parameters $a$, and each parameter $a_i$ appears in precisely one column (hence the term \emph{vertical}), a coefficient matrix $L\in\CC^{d\times n}$ with $d=n-s$, and a matrix $M\in\ZZ^{n\times r}$ whose columns are the exponent vectors of the monomial vector $x^M$. 

Vertically parametrized systems allow for linear dependencies among coefficients appearing in front of the same monomial, and naturally appear in optimization \cite[Example~1.1]{FeliuHenrikssonPascual2025vertical} and the study of $A$-discriminants \cite{BihanBound,FeliuFerrerTelek2025}. A main application lies in (bio)chemistry, where the positive zeros of augmented vertically parametrized systems (which have real coefficients) are the steady states of reaction networks, as we explain in \Cref{sec:reaction_networks} in more detail. 
Two simple examples of systems appearing in these applications are the following:

\begin{enumerate}[label=(\roman*)]
    \item The critical points in $(\CC^*)^2$ of the polynomial 
$f(x)=a_{1} x_{1}^2 + a_{2} x_{1} x_{2} + a_{3} x_{1}^3 x_{2}^2 + a_{4} x_{1}^3 x_{2}^3$
are the zeros of the vertically parametrized system
\begin{equation}
\label{eq:critical_point_system}
    F= \big(2 a _{1} x _{1}^2  + a _{2} x _{1} x _{2}+  3 a _{3} x _{1}^3 x _{2}^2 + 3 a _{4} x _{1}^3 x _{2}^3 ,\  
 a _{2} x _{1} x _{2}+ 2 a _{3} x _{1}^3 x _{2}^2+3 a _{4} x _{1}^3 x _{2}^3  \big)\,,
\end{equation}
with defining matrices
\[C={\footnotesize\begin{bmatrix}2a_1 & a_2 & 3a_3 & 3a_4\\ 0 & a_2 & 2a_3 & 3a_4 \end{bmatrix}}\,,\qquad M={\footnotesize\begin{bmatrix}2 & 1 & 3 & 3\\ 0 & 1 & 2 & 3 \end{bmatrix}}\,.\]
\item As a running example,  we will use the following augmented vertically parametrized system, which describes the steady states of the phosphorylation network in \eqref{eq:1-site_network}:
\begin{equation}\label{eq:1-site_equations}
\begin{aligned}
F & = \big(
a_{3} x_{5} - a_{4} x_{2} x_{4} + a_{5} x_{6}, \
a_{1} x_{1} x_{3} - a_{2} x_{5} - a_{3} x_{5}, \
a_{4} x_{2} x_{4} - a_{5} x_{6} - a_{6} x_{6},
\\  &  \hspace{5cm}
x_{3} + x_{4} + x_{5} + x_{6} - b_{1}, \
x_{1} + x_{5} - b_{2}, \ 
x_{2} + x_{6} - b_{3}\big)
\end{aligned}
\end{equation}
with defining matrices
\begin{equation}
\label{eq:minimal_presentation_1-site}
C ={\footnotesize \begin{bmatrix}0 & a_{3} & -a_{4} & a_{5}\\  a_{1} & -a_{2}-a_{3} & 0 & 0\\ 0 & 0 & a_{4} & -a_{5}-a_{6}\end{bmatrix}},
\:\:\:
M ={\footnotesize \begin{bmatrix}1 & 0 & 0 & 0\\ 0 & 0 & 1 & 0\\ 1 & 0 & 0 & 0\\ 0 & 0 & 1 & 0\\ 0 & 1 & 0 & 0\\  0 & 0 & 0 & 1\end{bmatrix}},
\:\:\:
L={\footnotesize\begin{bmatrix}0 & 0 & 1 & 1 & 1 & 1\\1 & 0 & 0 & 0 & 1 & 0\\0 & 1 & 0 & 0 & 0 & 1\end{bmatrix}}\, . 
\end{equation}
\end{enumerate}

Our first objective is to determine the maximal (or, equivalently, the generic) number of nondegenerate zeros in $(\CC^*)^n$ for parameter values in $\CC^{m+d}$. For \eqref{eq:critical_point_system}, a simple computation gives that this generic root count is three (which is lower than Bernstein's mixed volume bound, which is five). For \eqref{eq:1-site_equations}, we will later see that the generic root count is also three (which coincides with the mixed volume).

Knowing the generic root count is valuable for several purposes. First of all, it provides an upper bound on the number of nondegenerate zeros for any choice of parameters, which is often of interest in applications.
Second, it is useful for finding all zeros through numerical algebraic geometry: it provides a rigorous stopping condition for monodromy solving \cite{BBCHLS2023}, as well as a certificate that all solutions have been found when combined with interval arithmetic techniques \cite{BreidingRoseTimme23}. Furthermore, since any polynomial system can be seen as a specialization of a vertically parametrized system (see \Cref{rem:any_pol}), the generic root count of this system can be used as a general-purpose upper bound on the number of nondegenerate zeros, which is at most equal to, and often lower than, the commonly used BKK bound. 
In particular, this can be used to compute an upper bound on the degree of any complete intersection variety; see \Cref{cor:generic_degree}.

In this paper, we use \emph{tropical geometry} (reviewed in Sections~\ref{subsec:tropical_varieties}\nobreakdash-\ref{subsec:stable_intersection}) to compute the maximal number of nondegenerate complex zeros exactly.
The idea of counting roots via tropical geometry goes back to \cite{HuberSturmfels95}, which gives a tropical proof of Bernstein's theorem (without explicitly using tropical terminology). This approach has the advantage of also providing homotopies for numerically finding the solutions. 
The first ideas toward extending this perspective to systems with coefficient dependencies were given in \cite{LeykinYu2019} in the case of \emph{horizontally parametrized} systems. 
This was generalized in \cite{HelminckRen22}, which laid some of the foundations for finding root counts using tropical intersection theory, reembeddings, tropical flatness and Berkovich analytifications. 
These foundations were then used to obtain specialized root counts for certain horizontal systems \cite{HoltRen2023} and graph realization systems \cite{ClarkeDewarTrippGreenMaxwellNixonRenSmith2026}. 
The problem of constructing homotopies from tropical intersection data was further explored in \cite{HelminckHenrikssonRen2024}.

Here, we take a new and more elementary approach of computing root bounds using tropical geometry, with the classical transverse intersection theorem as the key theoretical tool. Exploiting this, we provide the generic root count of augmented vertically parametrized systems, as well as a new simplified formula \eqref{eq:purely_vertical_intro}  for the purely vertical case.

One of the key steps in both our and previously developed root counting strategies is to perform a \emph{toric reembedding}, and turn the problem into counting the number of intersections of a toric variety and a linear space. This is a classical strategy in the geometric study of polynomial systems,  e.g., \cite[Chapter 6]{gelfand1994discriminants}, \cite[Section~3.2]{Sottile-book}. In our setting, these varieties are $V(y-x^M)$ and $V(C_ay,Lx-b)$ in $(\CC^\ast)^{r+n}$, for auxiliary variables $y=(y_1,\ldots,y_r)$. To us, they have two advantages:  they are easy to tropicalize, and the vertically parametrized structure guarantees that the tropicalizations intersect transversely for generic Puiseux series parameters. This idea is illustrated in \Cref{fig:root_counting_strategy}. 

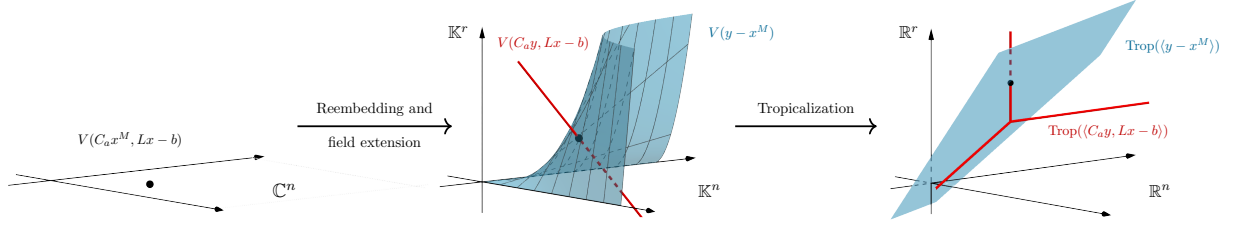
\begin{figure}[t]
\centering
\scalebox{1.7}{
\hspace{-0.5cm}
\begin{tikzpicture}
  \node[anchor=south west, inner sep=0] (img) at (0,0) {\includegraphics[width=10cm]{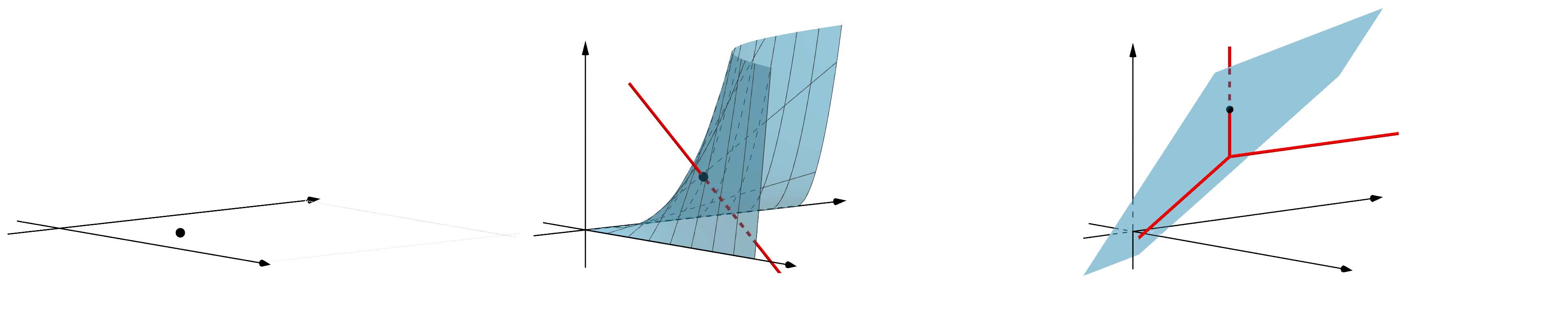}};
  \begin{scope}[x={(img.south east)}, y={(img.north west)}]

    \node at (0.22, 0.26) {\scalebox{0.39}{$\CC^n$}};
    \node at (0.10, 0.45) {\scalebox{0.28}{$V(C_a x^M,Lx-b)$}};
    
    \draw[-{>[scale=0.7]}, inner sep=0.2em] (0.23,0.5) -- (0.35,0.5)  node[midway, above] {\scalebox{0.3}{Reembedding and}} node[midway, below] {\scalebox{0.3}{field extension}};
    
    \node at (0.355, 0.85) {\scalebox{0.35}{$\mathbb{K}^r$}};
    \node at (0.55, 0.26) {\scalebox{0.35}{$\mathbb{K}^n$}};
    \node[color=darkred] at (0.421, 0.81) {\scalebox{0.28}{$V(C_ay,Lx-b)$}};
    \node[color=MidnightBlue] at (0.575, 0.85) {\scalebox{0.28}{$V(y-x^M)$}};
    
    \draw[-{>[scale=0.7]}, inner sep=0.2em] (0.57,0.5) -- (0.68,0.5)  node[midway, above] {\scalebox{0.3}{Tropicalization}};
    
    \node at (0.705, 0.85) {\scalebox{0.35}{${\RR^r}$}};
    \node at (0.90, 0.26) {\scalebox{0.35}{$\RR^n$}};
    \node[color=darkred] at (0.86, 0.48) {\scalebox{0.28}{$\Trop(\langle C_ay,Lx-b\rangle)$}};
    \node[color=MidnightBlue] at (0.91, 0.8) {\scalebox{0.28}{$\Trop(\langle y-x^M\rangle )$}};
  \end{scope}
\end{tikzpicture}\hspace{-4em}
}
\caption{Schematic illustration of our root counting strategy. We first reembed the system as an intersection of a toric variety (blue) and a linear space (red) over the field $\mathbb{K}$ of complex Puiseux series. After tropicalization, roots can be counted using polyhedral geometry, and genericity can be certified by transversality.}
\label{fig:root_counting_strategy}
\end{figure}

\begin{theoremalphabetic}[{\Cref{thm:monomial_reembedding_bounds}(i), \Cref{thm:bound_purevertical}}]
The maximal number of nondegenerate zeros in $(\CC^*)^n$ of \eqref{eq:augmented_vertically_parametrized_intro} is given by the tropical intersection number
\begin{equation}
\label{eq:complex_bound_intro}
\rowspan([\,M \mid \id_n\,])\cdot\Trop(\langle\, C_a y,\, Lx-b\,\rangle)
\end{equation}
in $\RR^{r+n}$, 
for $(a,b)\in\CC^{m}\times\CC^{d}$ chosen according to explicit genericity conditions. When $d=0$, 
the formula 
\eqref{eq:complex_bound_intro} can be simplified to 
\begin{equation}
\label{eq:purely_vertical_intro}
\deg(\phi_M)\cdot\big(\rowspan(M)\cdot \Trop(\langle\,C_a y\,\rangle)\big)\,,
\end{equation}
where $\deg(\phi_M)$ is the degree of the monomial map $\phi_M\colon (\CC^*)^n\to (\CC^*)^r, x \mapsto x^M$.
\end{theoremalphabetic}

Our second objective is to obtain lower bounds on the maximal possible number of nondegenerate positive zeros when the system has real coefficients. It turns out that this can be achieved by restricting the tropicalizations in \eqref{eq:complex_bound_intro} to the associated \emph{positive tropicalization}, in the sense of \cite{Alessandrini2007,JellScheidererYu}. We use the positive analog of the transverse intersection theorem from \cite{RoseTelek24}, to obtain the following result.

\begin{theoremalphabetic}[{\Cref{thm:monomial_reembedding_bounds}(ii)}]
If \eqref{eq:augmented_vertically_parametrized_intro} has real coefficients, then
the maximal number of nondegenerate positive zeros  for $(a,b)\in \RR^m_{>0}\times \RR^d$ is at least 
the number of points in the intersection
\[\big(\rowspan([\,M\mid\id_n\,])+(h,\mathbf{0}_{n})\big)\cap \Trop^+(\langle C_a y,Lx-b\rangle)\]
for any 
$(a,b)\in \RR^m_{>0}\times \RR^d$ and $h\in\RR^r$ chosen according to explicit genericity conditions. 
\end{theoremalphabetic}

Proving these bounds and building up the necessary background is the subject of \Cref{sec:tropical_bounds}. 
In particular, as we aim at a broad audience including people working on applications, we take an expository approach and explain the basic objects and results we need from tropical geometry.

After this, \Cref{sec:improvements} is devoted to two special structures that enable computational speedups, and which often  arise in reaction network theory:
\begin{itemize}
    \item \textbf{Cotransversality}: If the matroids of the matrices $C_a$ and $[\,L\mid -b\,]$ for generic parameters are \emph{cotransversal}, one can apply a linear transformation to the system \eqref{eq:augmented_vertically_parametrized_intro} such that the maximal number of nondegenerate complex zeros is the mixed volume of the resulting system. This is the content of \Cref{thm:cotransversal}.
    \item \textbf{Toricity:} If the positive zeros of the vertical part $C x^M$ of \eqref{eq:augmented_vertically_parametrized_intro} have \emph{toric structure} in the sense of \cite{FeliuHenriksson2024}, the maximal number of nondegenerate positive zeros is bounded from above and below by tropical intersection numbers in $\RR^n$ as in \Cref{thm:toric_root_bound}.
\end{itemize}

\medskip

In \Cref{sec:reaction_networks}, we discuss the application to reaction networks, recall the notion of \emph{steady state degree}, and show that for a well-studied family of multi-site phosphorylation networks, the steady state degree can be computed via a mixed volume computation.
Finally, in \Cref{sec:implementation},  we introduce our Julia package \texttt{VerticalRootCounts.jl}, built on the computer algebra system \texttt{OSCAR} \cite{OSCAR-book}, which provides an implementation of the bounds presented in the paper.

\subsection*{Acknowledgments}
EF and OH were supported by the European Union under the Grant Agreement No.\ 101044561, POSALG\footnote{\scriptsize Views and opinions expressed are those of the
authors only and do not necessarily reflect those of the European Union or European Research
Council (ERC). Neither the European Union nor ERC can be held responsible for them.}. 
PH was supported by the JSPS Postdoctoral Fellowship with ID No.\ 23769 and KAKENHI 23KF0187 as a Postdoctoral Fellow at Tohoku University and the University of Tsukuba. 
YR is supported by UKRI grant MR/Y003888/1 and EPSRC grant EP/Y028872/1.
BS is supported by the Swedish Research Council grant 2022-04224.  MLT was partially supported by the project 
101183111-DSYREKI-HORIZON-MSCA2023-SE-01. 

\section{Parametrized polynomial systems}
\label{sec:networks_and_vertical_systems}

\subsection{Preliminaries}\label{subsec:preliminaries}
In this section, we gather basic notions we use throughout the paper.

\smallskip
\paragraph{\bf Notation. }
For a field $\FF$, we let $\FF^*\coloneqq \FF \setminus \{0\}$. For a set $S$, we denote its cardinality by $\#S$. 
We let $[n]= \{1,\dots,n\}$.  For a vector $v\in \FF^n$, the support is denoted by $\supp(v)\coloneqq \{i\in [n] \mid v_i\neq 0\}$, and the diagonal matrix with the entries of $v$ on the diagonal is denoted by $\diag(v)$. We let $\star$ denote the Hadamard product or componentwise product of two vectors. 
For a $k\times n$ matrix $N$ and a subset $J \subseteq [n]$, we write $N_J$ for the submatrix of $N$ given by the columns with index in $J$. We denote by $\id_n$ the identity matrix of size $n$, and $\mathbf{0}_n$ the zero vector with $n$ entries.

\smallskip
\paragraph{\bf Monomial maps. }
For $h\in \QQ^r$ and $t$ in some field, we let $t^h=(t^{h_1},\dots,t^{h_r})$. 
For a matrix $M=(m_{ij}) \in \ZZ^{n\times r}$ and $x\in (\CC^*)^n$, we let $x^M$ be the vector whose $j$-th entry equals $\prod_{i=1}^n x_i^{m_{ij}}$. We consider the monomial map
\[ \phi_M\colon (\CC^*)^n \rightarrow (\CC^*)^r, \quad x\mapsto x^M\] 
and when $M$ has full row rank, we let $\deg(\phi_M)$ denote its  degree.
The (very affine) toric variety associated with $M$ is   $\im(\phi_M)$. For a vector $c\in (\CC^*)^r$, the corresponding scaled toric variety is 
\[c\star\im(\phi_M)=\{ c\star x^M \mid x\in (\CC^*)^n\}\,.\]
The dimension of $c\star\im(\phi_M)$ is $\rank(M)$.
We also consider the positive toric variety 
\[\mathcal{T}_M\coloneqq\im(\phi_M)\cap\RR^r_{>0}=\phi_M(\RR^n_{>0})\,,\]
and its scaled analog $c\star\mathcal{T}_M$ for $c\in\RR^r_{>0}$.  This set is Zariski dense in the complex analog, in the sense that
\[I(c\star \mathcal{T}_M)=I(c\star\im(\phi_M))\,,\]
where $I(\cdot)$ denotes the vanishing ideal in the Laurent polynomial ring $\CC[y_1^\pm,\ldots,y_r^\pm]$. This is a binomial ideal, generated by $y^u-c^u$ for $u$ in any basis of the integer kernel $\ker_{\ZZ}(M)$, and in particular it is a complete intersection ideal.

\smallskip
\paragraph{\bf Zeros of polynomial systems. }
For a polynomial system $F$ over $\CC$, 
we denote by $V_{\CC^*}(F)$   the variety of zeros  over the complex torus. If $F$ has real coefficients,  $V_{>0}(F)$ denotes the set of zeros of $F$ with entries in $\RR_{>0}$. 

A zero $x^*$ of a polynomial system $F$ is \term{nondegenerate} if the Jacobian matrix of $F$ at $x^*$ has full row rank. Analogously, for a complete intersection polynomial ideal $I$, we define a zero to be nondegenerate if it is nondegenerate with respect to any minimal generating set $F$ of $I$.

\smallskip
\paragraph{\bf Specialization of parametrized polynomials. }
For parametrized polynomial systems $F\in\CC[p][x^\pm]^s$
with variables $x=(x_1,\dots,x_n)$ and parameters $p=(p_1,\dots,p_k)$,  we denote by $F_p$ the specialization of $F$ at the point $p\in \CC^k$. Similarly, for an ideal $I \subseteq \mathbb{C}[p][x^\pm]$ or a matrix~$T\in \mathbb{C}[p]^{s\times r}$, we denote their specializations at $p$ by $I_p$ and $T_p$, respectively.

\smallskip
\paragraph{\bf Generic root counts and nondegeneracy. }
For a square parametrized polynomial system $F\in\CC[p][x^\pm]^n$ with variables $x=(x_1,\dots,x_n)$ and parameters $p=(p_1,\dots,p_k)$, it holds that the specialization $F_p$ has at most finitely many nondegenerate zeros for each $p\in\CC^k$. Furthermore, the parameter continuation theorem (see, e.g., \cite[Theorem~7.1.1]{SommeseWampler05} or \cite{BorovikBreiding2025}) tells us that there is a nonempty Zariski open subset $\mathcal{V}\subseteq\CC^{k}$
of parameter values for which the number of nondegenerate zeros is constant.
We call the generic number of nondegenerate zeros the \term{generic root count} of the system, and denote it  $\grc(F)$. Alternatively, it is characterized by
\begin{equation*}\label{eq:grc}
\grc(F)=\max_{ p\in\CC^k}\ \#\big\{x\in (\CC^*)^n \mid x  \text{ is a nondegenerate zero of $F_p$}\big\}\,.
\end{equation*}
For parametrized complete intersection ideals $I\subseteq\CC[a][x^\pm]$ with generically finite zero sets, we analogously define $\grc(I)$ as $\grc(F)$ for any minimal generating set $F$ of $I$.

We will be concerned with square parametric systems $F\in\CC[p][x^\pm]^n$ with \term{generically nondegenerate zeros}, in the sense that there is a Zariski open subset $\mathcal{W}\subseteq\CC^k$ such that all zeros of $F_p$ are nondegenerate for $p\in\mathcal{W}$. 
For instance, this holds for all systems $F$ such that the incidence variety $\{(p,x)\in\CC^k\times(\CC^*)^n\mid F_p(x)=0\}$ is irreducible and $F$ has at least one nondegenerate zero for some choice of parameters \cite[Theorem~2.15]{FeliuHenrikssonPascual2025vertical}. 

If $F$ has generically nondegenerate zeros, then $\grc(F)$ is also the generic (or equivalently, the maximal) number of isolated points of $V_{\CC^*}(F_p)$ for $p\in\CC^k$, and our notion of generic root count agrees with that in \cite{HelminckRen22}. Furthermore, for any $p\in\CC^k$ such that $\# V_{\CC^*}(F_p)=\grc(F)$, it holds that all zeros of $F_p$ are nondegenerate.

Note that the properties of generic root counts and generic nondegeneracy are preserved under field extension; if $F\in\CC[p][x^\pm]^n$ has generically nondegenerate zeros, then for any algebraically closed field extension $\CC\subseteq \mathbb{F}$, there exists a Zariski open dense set $\mathcal{V}\subseteq \mathbb{F}^k$ such that for $p\in\mathcal{V}$, the specialized system $F_p$ has $\grc(F)$ many zeros, all of which are nondegenerate.

\begin{remark}
\label{rmk:not_bound_of_isolated_solutions}
When $F$ has real coefficients, the generic root count does not necessarily give an upper bound on the number of \emph{isolated} real zeros; it only bounds the number of \emph{nondegenerate} zeros.
Inspired by \cite[Example~1.8]{bezout:real}, we may consider the
parametric system 
\begin{align*}
    F=&  \big( a_1 x_{1}^{6}+ a_2 x_{2}^{6}+ a_3 x_{1}^{5}+a_4 x_{2}^{5}+a_5 x_{1}^{4}+a_6 x_{2}^{4}+ a_7 x_{1}^{3}+ a_8 x_{2}^{3}+ a_9 x_{1}^{2}+a_{10} x_{2}^{2}\\ & \hspace{2cm} +a_{11}\, x_{1}+a_{12}\, x_{2}+a_{13}, \  a_{14}\,x_2+a_{15}\,x_3 + a_{16}, \ a_{17}\,x_2+a_{18}\,x_3 + a_{19}\big)\,.
\end{align*}
The generic root count is six, but specializing to 
\[ a^*=(1,1,-12,-12,58,58,-144,-144,193,193,-132,-132,72,1,1,1,1,1,1)\]
gives
\[F_{a^*}=\Big((x_1-1)^2(x_1-2)^2(x_1-3)^2 + (x_2-1)^2(x_2-2)^2(x_2-3)^2, \ x_2+x_3+1, \ x_2+x_3+1\Big)\,,\]
which has infinitely many complex solutions and nine real solutions (all of which are degenerate). Hence, the number of isolated real solutions exceeds the generic root count.
\end{remark}

\subsection{Augmented vertically parametrized systems}
\label{subsec:root_counts_and_augmented_vertical_systems}
The parametrized systems that we consider in this article are \term{(linearly) augmented vertically parametrized systems}, which were introduced and studied in \cite{FeliuHenrikssonPascual2025vertical}. These are systems that can be written in the form
 \begin{align}
    \label{eq:vertical}
     F = (C x^M, 
     \: L x - b)\, \in \CC[a,b][x^\pm]^n
    \end{align}
    with 
    \begin{itemize}
        \item a parametric coefficient matrix $C\in \CC[a]^{s\times r}$ that is \term{vertical}, in the sense that each entry is either zero or a homogeneous linear form in the parameters $a=(a_1,\dots,a_m)$, and each $a_i$ appears in precisely one column of $C$; we furthermore require $C$ to 
        generically 
        be of full row rank;  
        \item an exponent matrix $M\in \ZZ^{n\times r}$;
        \item a fixed coefficient matrix $L\in \CC^{d\times n}$ of full row rank and parametric constant terms $b=(b_1,\dots,b_d)$. 
         \end{itemize}
The systems with empty $L$  (i.e., $d=0$)  were introduced in \cite{HelminckRen22} and are simply said to be (purely) \term{vertically parametrized}.

It was proven in \cite[Theorem~3.7]{FeliuHenrikssonPascual2025vertical} that augmented vertically parametrized systems have generically radical ideals and nondegenerate zeros. Hence, it follows from \Cref{subsec:preliminaries} that there exists a nonempty Zariski open subset   $\mathcal{V}\subseteq\CC^m\times\CC^d$ such that for $(a,b)\in \mathcal{V}$, the specialized system $F_{a,b}$ has exactly $\grc(F)$ many zeros, all of which are nondegenerate.

If $F\in \RR[a,b][x^\pm]^n$ (i.e., the coefficients are real), we also consider the number of nondegenerate zeros in $\RR^n_{>0}$, and define the \term{maximal positive root count} as
\begin{equation*}\label{eq:max}
\mrc_{>0}(F) \coloneqq \max_{(a,b)\in \RR^m_{>0}\times L(\RR^n_{>0})}\ \#\big\{ 
x\in \RR^n_{>0} \mid x  \text{ is a nondegenerate zero of $F_{a,b}$}\big\}\, . 
\end{equation*}
It is immediate from these definitions that $\grc(F)\geq\mrc_{>0}(F)$.
Restricting to positive $a\in\RR^m_{>0}$ allows us to study systems with \emph{signed support}, as described in \cite[Corollary 3.17]{FeliuHenrikssonPascual2025vertical} (see also \cite[Remark~2.3]{FeliuHenriksson2024} and \cite{FeliuFerrerTelek2025}). 
The augmented vertically parametrized systems arising in our motivating application in reaction network theory also require $a$  to be positive (see \Cref{sec:reaction_networks}). The restriction to $b\in L(\RR^n_{>0})$ comes from the fact that $F_{a,b}$ lacks positive zeros outside this locus.
 
For any augmented vertically parametrized system $F$, there is a unique (up to reordering) choice of $C$ and $M$ such that $M$ has no repeated columns. We call this pair $(C,M)$ the \term{minimal presentation} of $F$. 
In general, we do not impose that the columns of $M$ are different, and hence different pairs $(C,M)$ may exist, as the next example illustrates.

\begin{example}[Running example]
\label{ex:running_matrices}
For the running example \eqref{eq:1-site_equations}, the minimal presentation was given in \eqref{eq:minimal_presentation_1-site}. 
We might also express $F$ with the following pair of matrices, where now each column of $C$ depends exactly on one parameter and hence $M$ has repeated columns: 
\begin{equation}\label{eq:matrices_vertical}
C={\small\begin{bmatrix} 
0 & 0 & a_3 & -a_4 & a_5 & 0 \\ a_1 & -a_2 & -a_3 & 0 & 0 & 0 \\ 0 & 0 & 0 & a_4 & -a_5 & -a_6\end{bmatrix}},\qquad
M={\small \begin{bmatrix}1 & 0 & 0 & 0 & 0 & 0 \\ 0 & 0 & 0 & 1 & 0 & 0 \\
1 & 0 & 0 & 0 & 0 & 0 \\ 0 & 0 & 0 & 1 & 0 & 0 \\0 & 1 & 1 & 0 & 0 & 0 \\ 0 & 0 & 0 & 0 & 1 & 1 \end{bmatrix}}\,. \qedhere
\end{equation}
\end{example}

The vertical part $C x^M$ of any augmented vertically parametrized system $F$ can be written in such a way that $C$ has $m$ columns and each column depends on only one parameter, as we had in  
 \eqref{eq:matrices_vertical}. This leads to an expression of the form 
\begin{equation}\label{eq:parameter_separating}
 C x^M = \overbar{C}\diag(a) \, x^{\overbar{M}}=\overbar{C}(a\star  x^{\overbar{M}}) 
\end{equation}
for some $\overbar{C}\in \CC^{s\times m}$ without zero columns and $\overbar{M}\in \ZZ^{n\times m}$.
We refer to the pair $(\overbar{C}\diag(a),\overbar{M})$ as the \term{parameter-separating presentation} of $F$, which is unique up to reordering of the columns. 
This presentation is convenient for some purposes, as it allows us to specify $F$ through two parameter-free matrices $\overbar{C}$ and $\overbar{M}$. For instance, 
it was proven in \cite[Theorem~3.7]{FeliuHenrikssonPascual2025vertical}  that $\grc(F)$ is zero if and only if 
\begin{equation}
\label{eq:rank_condition}
    \rank \begin{pmatrix}
         \overbar{C}  \diag(w) \overbar{M}^\top  \diag(h) \\ L
     \end{pmatrix}<n\quad\text{for all $(w,h)\in \ker(\overbar{C}) \times (\CC^*)^n$}\,.
\end{equation}

\begin{remark}\label{rem:any_pol}
Any  $n$-variate polynomial system $G\in \CC[x^\pm]^n$ can be seen as the specialization of a vertically parametrized system: writing $G=Px^B$ for $B\in \ZZ^{n\times m}$ and $P\in \CC^{n\times m}$, we have $G=F_{a^*}$ for the vertically parametrized system $F=P (a \star x^B)$ with $a=(a_1,\dots,a_m)$ specialized at the parameter $a^*=(1,\dots,1)$. Therefore, the number of nondegenerate zeros of $G$ in $(\CC^*)^n$ is at most $\grc(F)$ and the number of nondegenerate positive zeros of $G$ is at most $\mrc_{>0}(F)$. 
\end{remark}

\subsection{The toric reembedding}\label{subsec:reembedding}
Consider an augmented vertically parametrized system $F=(Cx^M, Lx-b)$ as in \eqref{eq:vertical}. 
Its \term{toric reembedding} is the system
\begin{align}
\label{eq:monomial_reembedding}
   \widehat{F} \coloneqq \big( y -  x^{M} , C y  ,   Lx - b \big) \in \CC[a,b][y^\pm,x^\pm]^{r+n} \, 
\end{align}
with $y=(y_1,\dots,y_r)$, which can be 
decomposed into a linear and a binomial part
\begin{equation} \label{eq:monomial_linbin}
I^{\bin}\coloneqq \langle y - x^{M}  \rangle \subseteq \CC[y^\pm,x^\pm]\,, \qquad  I^{\lin} \coloneqq \langle C y  , Lx - b  \rangle \subseteq \CC[a,b][y^\pm,x^\pm]  \,.
\end{equation}
A simple computation shows that for fixed $(a,b)\in \CC^{m}\times\CC^d$, the rings $\CC[y^\pm,x^\pm]/\langle \widehat{F}_{a,b}\rangle$ and $\CC[x^\pm]/\langle F_{a,b}\rangle$ are isomorphic, and we have a multiplicity-preserving  bijection 
\begin{align}\label{eq:monomial_correspondence}
V_{\CC^*}(F_{a,b}) \to  V_{\CC^*}(I^{\bin}+I^{\lin}_{a,b}),\qquad x \mapsto  (x^M,x)\,.
\end{align}
Thus, to study nondegenerate zeros of an augmented vertically parametrized system, it is enough to understand the nondegenerate zeros of 
$I^{\bin}+I^{\lin}_{a,b}$. 
This is a crucial observation for the root counting technique presented in \Cref{subsec:root_bounds}.

\begin{example}
For the system \eqref{eq:1-site_equations} with the minimal presentation \eqref{eq:minimal_presentation_1-site}, the binomial and linear parts of the toric reembedding are given by
\begin{align*}
I^{\bin}&=\langle y_1 - x_1x_3, \  y_2 - x_5,\   y_3 - x_2x_4, \  y_4 - x_6\rangle,\\
I^{\lin}&=\langle a_3y_2 - a_4y_3 + a_5y_4, \ 
a_1y_1 -(a_2+a_3)y_2, a_4y_3 - (a_5+a_6)y_4, \\ & \hspace{3cm} 
x_{3} + x_{4} + x_{5} + x_{6} - b_{1},\    x_{1} + x_{5} - b_{2},   x_{2} + x_{6} - b_{3}\rangle\,. \qedhere
\end{align*}
\end{example}

\section{Bounds via tropical geometry}
\label{sec:tropical_bounds}

\subsection{Tropical varieties}\label{subsec:tropical_varieties}
In this subsection, we give an overview of the key concepts from tropical geometry that are required for our root counting techniques. 
Our first step will be to extend the base field  from $\CC$ to the field of \term{complex Puiseux series}, defined as
\begin{equation*}
\label{eq:puiseuxSeriesDefinition}
\mathbb{K}\coloneqq \CCt= \Big\{\sum_{i=k}^\infty c_{i/N} \,  t^{i/N} \mid k\in\ZZ,\;\; N\in\ZZ_{>0},\;\; c_{i/N} \in \CC,\;\; c_{k/N} \neq 0 \Big\}\cup\{0\}\,,
\end{equation*}
for which we have the valuation map
\begin{equation*}
\val\colon\quad \mathbb{K}^\ast\rightarrow\RR\, ,\qquad \sum_{i=k}^\infty c_{i/N}\, t^{i/N}\mapsto \ k/N \,,
\end{equation*}
which we also extend to vectors $z\in (\mathbb{K}^\ast)^n$ by setting $\val(z)=(\val(z_1),\dots,\val(z_n))$.

Let $\mathbb{K}[x^{\pm}]$ be the Laurent polynomial ring over $\mathbb{K}$ in $n$ variables. 
We write $V(I)$ for the set of zeros in $(\mathbb{K}^*)^n$ for an ideal $I\subseteq \mathbb{K}[x^\pm]$. For ideals $I\subseteq \CC[x^\pm]$, we implicitly consider the  extension to $\mathbb{K}[x^\pm]$.
Note that $V(I)$ and 
$V_{\CC^*}(I)$
coincide if $I\subseteq \CC[x^\pm]$ is zero-dimensional. 

The \term{tropical variety} of an ideal  $I\subseteq \mathbb{K}[x^\pm]$ is defined as the set
\begin{equation}
\label{eq:tropicalization}
\Trop(I)\coloneqq \overline{\left\{ \val(z) \mid z\in V(I) \right\}}\subseteq\RR^n,
\end{equation}
where the closure is taken with respect to the Euclidean topology on $\RR^n$. 
This set is equipped with the structure of a \emph{polyhedral complex} described in \cite[Section~3.2]{MaclaganSturmfels15}, which we denote by $\Trop(I)$ as well. This complex has the property that its support (i.e.,  the union of all its cells) is the right-hand side of \eqref{eq:tropicalization}, and all points $w$ that belong to the same cell determine the same initial ideal of $I$ \cite[Theorem 3.2.5, Proposition 3.2.8]{MaclaganSturmfels15}. 

When $I$ is pure-dimensional, \cite[Theorem~3.3.5,Theorem~3.4.14]{MaclaganSturmfels15} give that the dimension of $\Trop(I)$ as a polyhedral complex coincides with the Krull dimension of the ideal $I$. Additionally, $\Trop(I)$ has the structure of a \emph{weighted balanced polyhedral complex}, in the sense that each inclusion-maximal polyhedron $\sigma\in \Trop(I)$ has a  \emph{multiplicity} $\mult(\sigma)\in\ZZ_{>0}$, defined in \cite[Definition~3.4.3]{MaclaganSturmfels15}, that satisfies a technical balancing condition  from \cite[Definition~3.3.1]{MaclaganSturmfels15}.
Further details can be found in \cite[Section~3.2--3.4]{MaclaganSturmfels15}.

The following fundamental property, which follows from \cite[Proposition~3.4.8]{MaclaganSturmfels15}, allows us to use tropical geometry for counting zeros of zero-dimensional ideals. We denote the \term{degree} or \term{total multiplicity} of a zero-dimensional weighted balanced polyhedral complex $\Sigma$ by
\begin{equation*}\label{eq:total_multiplicity}
\deg(\Sigma):=\sum_{w\in \Sigma}\mult(w)\, .
\end{equation*}

\begin{theorem}[Zero-dimensional tropicalization]
\label{thm:zero-dimensional_case}
Let $I\subseteq \mathbb{K}[x^\pm]$ be an ideal with finitely many zeros that are all nondegenerate.
Then
\[ \deg(\Trop(I))=\#V(I) \, .\]
\end{theorem}

In principle, \Cref{thm:zero-dimensional_case} tells us that $\grc(F)$ for $F$ as in \eqref{eq:vertical} agrees with the total multiplicity of  
$\Trop(\langle F_{a,b}\rangle)$ for a generic choice of $(a,b)$. 
In practice, however, tropicalizations of zero-dimensional ideals might be computationally expensive (see, e.g., \cite{GorlachRenZhang2022}). Yet another downside of applying \Cref{thm:zero-dimensional_case} is that it does not come with a certificate that a given choice of $(a,b)$ is generic enough to give the generic root count $\grc(F)$. 

In what follows, we therefore make use of the tropical root counting strategy developed in \cite{LeykinYu2019,HelminckRen22,HoltRen2023,HelminckHenrikssonRen2024} of decomposing the system into pieces that are easier to tropicalize, and for which one obtains certificates of the genericity of the parameters. Specifically, we use the toric reembedding from \Cref{subsec:reembedding} and, to this end, we treat the case of tropicalizing binomial and affine ideals in \Cref{subsec:tropicalizing_toric_ideals,subsec:tropicalizing_affine_ideals} below. 
In \Cref{subsec:stable_intersection}, we discuss how sums of ideals correspond to intersections of tropical varieties, and finally, in \Cref{subsec:root_bounds}, we present our root counting formulas with explicit genericity conditions on the parameters.

We end the subsection by giving a brief introduction to the concept of \emph{positive} tropicalizations, which allows us to derive lower bounds on the maximal positive root count $\mrc_{>0}(F)$. The Puiseux series with  real  coefficients   form a real closed subfield of $\mathbb{K}$, denoted by $\RRt$, and   the multiplicative group of \term{positive real Puiseux series} is 
  \begin{equation*}
     \RRt_{>0}\coloneqq \Big\{ \sum_{i=k}^\infty c_{i/N} \, t^{i/N}\in \RRt \mid   c_{k/N} > 0 \Big\}\,.
     \end{equation*}
The \term{positive tropical variety} of an ideal $I\subseteq\RRt[x^\pm]$ is defined as
  \begin{equation}\label{eq:trop_pos}
    \Trop^+(I)\coloneqq\overline{\left\{ \val(z) \mid z\in V(I)\cap   \RRt_{>0}^n \right\}}
  \end{equation}
with closure taken with respect to the Euclidean topology on $\mathbb{R}^n$.
This set can be equipped with a polyhedral complex structure \cite[Theorem~4.10]{Alessandrini2007}, \cite[Theorem 3.1]{TropicalSpectrahedra} (which need not be a subcomplex of $\Trop(I)$). Here, we are only interested in $\Trop^+(I)$ as a set.

\subsection{Tropicalization of binomial ideals}
\label{subsec:tropicalizing_toric_ideals}
We consider prime binomial  ideals of the form $I=\langle x^B-c\rangle\subseteq \mathbb{K}[x^\pm]$ for an exponent matrix $B\in\ZZ^{n\times r}$ and vector $c\in(\mathbb{K}^*)^r$. Such ideals are precisely the ideals of scaled toric varieties in $(\mathbb{K}^*)^n$. 
By applying the valuation map to both sides of the equation $x^B=c$, one sees that the tropicalization of $I$ is simply a classical affine linear space. The weighted balanced polyhedral structure is described in  \cite[Lemma~2.7]{BLMM17}. We gather these results in the following lemma.

\begin{lemma}
\label{lem:binomial}
Let $I=\langle x^B-c\rangle\subseteq \mathbb{K}[x^\pm]$ be a prime binomial ideal with $B\in\ZZ^{n\times r}$ and $c\in(\mathbb{K}^*)^r$. Then, as a set, 
\[\Trop(I)=\left\{w\in\RR^n \suchthat
B^\top w=\val(c)\right\}\, .\]
As a weighted balanced polyhedral complex, it consists of a single polyhedron of multiplicity $1$. 
If additionally $c\in \RRt_{>0}^r$, it holds that
  \begin{equation*}
      \Trop^+(I) = \Trop(I)\,.
  \end{equation*}
\end{lemma}

\begin{example}
  \label{ex:polyhedralStructureBinomial}
   Applying \Cref{lem:binomial} to $I=\langle x_1\, x_2 -1\rangle$ we obtain (see \Cref{fig:affineLinearSpace}):
  \begin{equation*}
   \Trop(I)=  \Trop^+(I) =\Big\{ (w_1,w_2)\in\RR^2\mid w_1 + w_2 = 0\Big\} = \Span_{\RR}( (1,-1))\,. \qedhere
  \end{equation*}
\end{example}

\subsection{Tropicalization of affine ideals}
\label{subsec:tropicalizing_affine_ideals}
Next, we treat \term{affine ideals} of the form
\[ I=\langle Ax-c\rangle\subseteq \CC[x^\pm], \qquad A\in \CC^{k \times n}, \quad c\in \CC^k \, . \]
It is well known that the tropicalization is determined by the \emph{matroid} of the augmented coefficient matrix $[\,A\mid-c\,]$, which we briefly recall below. 
For a comprehensive introduction to matroid theory, Bergman fans, and tropical linear spaces, we refer to \cite[Section~4.2]{MaclaganSturmfels15} and  \cite{Oxley}. See also \cite[Section~3]{BonifaceDevriendtHosten2025} for an explicit description of tropical affine spaces. 

A \emph{matroid} $\mathcal{M}$ on an index set $[n]$ is a collection of subsets of $[n]$ that are declared to be \emph{circuits}, subject to certain axioms \cite[Section~1.1]{Oxley}.
In this paper, we are only concerned with \emph{linear} matroids, which can be defined by a matrix in the sense of the following definition. Note that the resulting matroid is the \emph{dual of the column matroid} of the matrix \cite[Section~2.1]{Oxley}; however, in what follows, we will simply refer to it as ``the'' matroid of the matrix. 
 
\begin{definition}
  \label{def:matroid}
  Let $A\in \mathbb{F}^{k\times n}$ be a matrix over a field $\mathbb{F}$. The \term{matroid} of $A$, denoted $\mathcal{M}^*[A]$, is the matroid on $[n]$ whose circuits are the supports of the vectors in $\rowspan(A)$ with minimal support.
  The set of circuits of $\mathcal{M}^*[A]$ is denoted by $\mathcal{C}^*[A]$.
\end{definition}

A more refined notion used for positive tropicalization is that of \emph{oriented matroids}, where the circuits are \emph{signed}, in the sense that  they are the disjoint union of two sets, identified respectively as the positive and negative parts. We refer to   \cite{BjornerLasVergnasSturmfelsWhiteZiegler:1999} for a general introduction.

\begin{definition}
   For a matrix $A\in\mathbb{R}^{k\times n}$,  the \term{oriented matroid} of $A$, denoted $\mathcal{M}_{\pm}^*[A]$, is defined by letting the signed circuits be the pairs $(C^+,C^-)$ obtained as
   \[ C^+ =\{ i\in[n] \mid v_i > 0\} \qquad  C^- = \{ i\in[n] \mid v_i < 0\} \, \]  
for some $v\in\rowspan(A)$ with $\supp(v)\in \mathcal{C}^*[A]$.  We let $\mathcal{C}_{\pm}^*[A]$ denote the set of signed circuits.
\end{definition}

\begin{figure}[t]
\noindent
  \centering
  \begin{minipage}[h]{0.4\textwidth}
    \centering
 \scalebox{0.75}{\begin{tikzpicture}
      \draw[darkred,-,line width=2pt] (-1.5,1.5) -- (1.5,-1.5);
    \draw[gray,dashed,line width=1pt] (0,-2) -- (0,2);
     \draw[gray,dashed,line width=1pt] (-2,0) -- (2,0);
    \node[black,anchor=center] (a) at (0,2.5) {$\Trop(\langle x_1\, x_2 -1\rangle)$};
    \node (w2) at (-0.3,1.7) {\small $w_2$};
    \node (w1) at (2,-0.3) {\small $w_1$};
   \end{tikzpicture}}
  \end{minipage}
  \begin{minipage}[h]{0.45\textwidth}
    \centering
  \scalebox{0.75}{\begin{tikzpicture}
  \node[black,anchor=center] (a) at (0,2.5) {$\Trop(\langle x_1+x_2-1\rangle)$};
    \draw[gray,dashed,line width=1pt] (0,-2) -- (0,2);
     \draw[gray,dashed,line width=1pt] (-2,0) -- (2,0);
       \node (w2) at (-0.3,1.7) {\small $w_2$};
    \node (w1) at (2,-0.3) {\small $w_1$};

    \draw[darkblue,-,line width=1pt] (0,0) -- ++(-1.6,-1.6);
    \draw[darkred,-,line width=2pt] (0,0) -- ++(2,0);
    \draw[darkred,-,line width=2pt] (0,0) -- ++(0,2);
    \coordinate (v0) at (-0.8,-0.4);
\end{tikzpicture}}
  \end{minipage}
  \caption{
  Left figure:  $\Trop(\langle x_1\, x_2-1\rangle)$  from \Cref{ex:polyhedralStructureBinomial}, which coincides with $\Trop^+(\langle x_1\, x_2-1\rangle)$. Right figure: 
  $\Trop(\langle x_1+x_2-1\rangle)$ from \Cref{ex:polyhedralStructureAffine}, with only the two red thicker rays defining $\Trop^+(\langle x_1+x_2-1\rangle)$.
  }
  \label{fig:affineLinearSpace}
\end{figure}
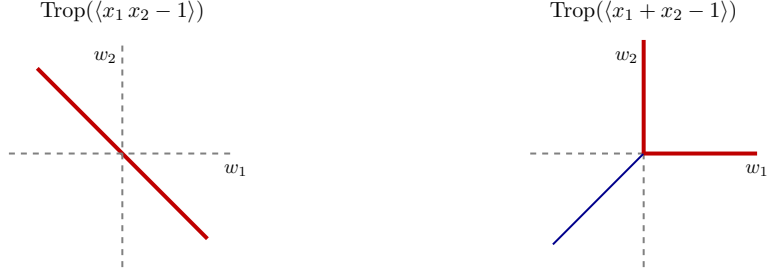

The matroid of a full row rank matrix is precisely determined by which of its maximal minors (also known as \emph{Plücker coordinates}) vanish \cite[Section~6.8]{Oxley}. Similarly, the oriented matroid is determined by the sign of the Plücker coordinates \cite[Section~8.1]{BjornerLasVergnasSturmfelsWhiteZiegler:1999}. In particular, for two matrices $A,B\in \FF^{k\times n}$ it holds that
\begin{align}
  \mathcal{M}^*[A] =   \mathcal{M}^*[B] & \quad \Leftrightarrow \quad \{J\in\textstyle\binom{[n]}{k}  \mid \det A_J = 0\}  =\{J\in\textstyle\binom{[n]}{k}  \mid \det B_J = 0\} \label{eq:matroids_agree}\, ,
  \\ 
   \mathcal{M}^*_\pm[A] =   \mathcal{M}^*_\pm[B]
   & \quad \Leftrightarrow \quad  \sgn(\det A_{J})=\sgn(\det B_{J}) \text{ for all }J\in\textstyle\binom{[n]}{k} \, . \label{eq:signed_matroids_agree}
\end{align}

An important example that will play a key role later is the \term{uniform matroid} of rank $k$ on $[n]$, for which all subsets of $[n]$ of size $k+1$ are circuits. Equivalently, this is the matroid $\mathcal{M}^*[A]$ for a matrix $A\in\QQ^{k\times n}$ for which all $k$-minors are nonzero.

We are now ready to describe the tropicalization of affine ideals. 

\begin{theorem}
\label{thm:TropLinearSpace}\label{thm:PosTropLinearSpace}
Let $I =\langle Ax - c\rangle \subseteq \CC[x^\pm]$ be an affine  ideal defined by  the matrix $A\in \CC^{k\times n}$  
of rank $k$ 
and the vector $c\in \CC^k$. Furthermore, let the map $\pi\colon\mathbb{R}^{n+1}\to\mathbb{R}^n$ denote the projection onto the first $n$ coordinates. Then, as a set,
\begin{multline*}
\Trop(I)  = \pi \big(\{ w \in \mathbb{R}^{n+1} \mid 
w_{n+1}=0 \text{ and for all } C \in \mathcal{C}^*[\,A\mid -c\,] \text{ it holds that}   \\  \textstyle\min_{\ell \in C} \, w_\ell = w_i = w_j\   \text{ for some }i,j\in C,\ i\neq j \} \big)\,.
\end{multline*}
As a weighted polyhedral complex, $\Trop(I)$ is an $(n-k)$-dimensional polyhedral fan, with all multiplicities equal to one. It is the coarsest of all polyhedral complexes supported on $\Trop(I)$.
 
If in addition $I  \subseteq \RR[x^\pm]$, then
\begin{align*}
\Trop^+(I)  = \pi \big(\{ w \in \mathbb{R}^{n+1} \mid 
& w_{n+1}=0 \text{ and for all } (C^+,C^-) \in \mathcal{C}^*_\pm[\,A\mid -c\,] \text{ it holds that} \\  & \hspace{5mm} \textstyle\min_{\ell \in C^+\cup C^-} \, w_\ell = w_i = w_j   \ \text{ for }i \in C^+ ,\  j \, \in C^- \} \big)\,.
\end{align*}
\end{theorem}
\begin{proof}
    The description of $\Trop(I)$ is a consequence of several statements in \cite{MaclaganSturmfels15}: Proposition~4.1.6 and  Proposition~2.6.1  to dehomogenize, Theorem 3.2.4 for the extension of $I$ to $\mathbb{K}[x^\pm]$, Definition~3.4.3 for the multiplicities being 1, and Section 4.2 for the polyhedral structure.
       The description of $\Trop^+(I)$ is given in \cite[Proposition~4.1]{ArdilaKlivansWilliams2006} and \cite[Proposition~3.1]{RoseTelek24}. 
\end{proof}

\begin{remark}\label{rmk:linear_ideal}
When $I$ is linear ($c=0$), the resulting tropical linear space is also known as the \emph{Bergman fan} of   $\mathcal{M}^*[A]$.    
In this case, $\Trop(I)$ is simply the set of vectors $\omega\in \RR^n$ for which the minimum is attained twice on the support of all circuits of $\mathcal{M}^*[A]$. An analogous simplification holds for $\Trop^+(I)$.
\end{remark}

\begin{example}
\label{ex:polyhedralStructureAffine}
Consider the affine  ideal $I=\langle x_1+x_2-1\rangle$ with defining matrix $[\,A \mid -c\,]$ equal to $[1\ \ 1\ \ -1]$. It is straightforward to see that 
\[\mC^*[\,A \mid -c\,]=\big\{\{1,2,3\}\big\} \quad \text{ and } \quad \mC^*_\pm[\,A \mid -c\,] = \big\{\,(\{1,2\},\{3\}),\: (\{3\},\{1,2\})\,\big\}\,.\]
Hence $\Trop(I)$ consists of the vectors $(w_1,w_2)\in\RR^2$ such that 
\[ w_1=w_2\leq 0\,\quad \text{or} \quad 0=w_1\leq w_2\,\quad \text{or} \quad0=w_2\leq w_1\,.\]
These   define three rays in $\RR^2$,  of which only the last two belong to $\Trop^+(I)$; see   \Cref{fig:affineLinearSpace}.
\end{example}

\begin{example}
\label{ex:BergmanFanRunning}
For $b\in \CC$, consider the affine  ideal $I=\langle A x - c \rangle $ 
with
\begin{equation*}
        [ A \, |\, -c] = \begin{bmatrix}
            1 & -1 &-1 & 0 &0 &0 \\
            0 & 0 & 0 & 1 & 1 & -b
        \end{bmatrix}\,.
    \end{equation*}
If $b\neq 0$, clearly $\{1,2,3\}$ and $\{4,5,6\}$ are the circuits of $\mathcal{M}^*[\, A \mid -c\,]$. 
    By \Cref{thm:TropLinearSpace}, $\Trop(I)$ consists of all $w\in \RR^5$ such that  
\begin{equation}\label{eq:matroid1}
w_1=w_2 \leq w_3 \quad\text{or} \quad w_1=w_3 \leq w_2 \quad\text{or} \quad w_2=w_3 \leq w_1 \,
\end{equation}
and  
\begin{equation}\label{eq:matroid2} 
w_4=w_5 \leq 0 \quad\text{or} \quad 0=w_4 \leq w_5 \quad\text{or} \quad 0=w_5\leq w_4 \,.   
\end{equation}
As a set, $\Trop(I)$ is the union of the $9$ maximal dimensional polyhedral cones
obtained by considering one constraint from \eqref{eq:matroid1} and one from \eqref{eq:matroid2}. These can be expressed as 
\[ \cone(e_i,-e_4-e_5)+ \mathcal{L}\,,  \quad \cone(e_i,e_4)+ \mathcal{L}\,, \quad \cone(e_i,e_5)+ \mathcal{L}\,, \qquad i=1,2,3\,, \]
where $e_1,\dots,e_5 \in \mathbb{R}^5$ are the standard basis vectors, and $\mathcal{L} = \Span_{\RR}((1, 1, 1, 0, 0))$ is the lineality space of each of the cones. 
When $b$ is a positive real number, the  set of signed circuits is 
\[ \mathcal{C}^*_\pm[\,A\mid -c\,] = \big\{ (\{1\},\{2,3\}), \; (\{2,3\},\{1\}),\; (\{4,5\},\{6\}), \; (\{6\},\{4,5\}) \big\}\, .\]
For $\Trop^+(I)$, we   remove the inequalities $w_2=w_3 \leq w_1$ and $w_4=w_5 \leq 0$ from 
\eqref{eq:matroid1} and   \eqref{eq:matroid2}, as the minimum  is attained at two indices both included in the positive or negative part of a signed circuit. 
Hence, $\Trop^+(I)$ equals the union of four polyhedral cones:
    \begin{equation*}
       \left( \cone(e_2,e_4) \cup \cone(e_2,e_5) \cup \cone(e_3,e_4) \cup \cone(e_3,e_5)\right)+ \mathcal{L}\,.\qedhere
    \end{equation*}
\end{example}

In our setting, we work with affine ideals where  $A$ and $c$ depend polynomially on some parameters $b=(b_1,\ldots,b_d)$. It is not hard to see that for generic choices of the parameters, the associated matroid $\mathcal{M}^*[\,A_b\mid -c_b]$ coincides with the \emph{generic matroid} $\mathcal{M}^*[\,A\mid -c]$ computed over the field $\CC(b)$ of rational functions in the parameters. 
In this situation, we let $\Trop(\langle Ax-c\rangle)$  refer to the tropicalization of $\langle A_bx-c_b\rangle$ for any such generic $b$. 
The following lemma, which is an immediate consequence of \eqref{eq:matroids_agree} and \eqref{eq:signed_matroids_agree}, gives an explicit description of this generic locus.

\begin{lemma}[Generic matroids]
\label{lem:genericMatroid}
Let $A\in \CC[b]^{k\times n}$ and $c\in\CC[b]^k$ for   $b=(b_1,\ldots, b_d)$, and suppose that the generic rank of $[\,A\,|\, -c\,]$ is $k$. 
\begin{enumerate}[label=(\roman*)]
    \item The matroid $\mathcal{M}^*[\,A\mid -c\,]$ over the field $\CC(b)$ coincides with the specialized matroid $\mathcal{M}^*[\,A_b\mid -c_b\,]$ if and only if $b$ is in the set
     \[ \mathcal{B}:= \left\{b\in\CC^d\mid \det [\,A| -c\,]_{J} = 0 \text{ for all $J\in\textstyle\binom{[n+1]}{k}$ such that } \det[\,A_b\mid -c_b\,]_{J} = 0 \right\}. \]
    In particular, $\Trop(\langle A_b\, x-c_b\rangle)$ is constant for $b\in \mathcal{B}$. 
    \item
    Suppose $A$ and $c$ have real entries. Then, for any fixed $b^*\in \RR^d$, it holds that $\mathcal{M}^*_{\pm}[\,A_b\mid -c_b\,]$ coincides with $\mathcal{M}^*_{\pm}[\,A_{b^*}\mid -c_{b^*}\,]$ if and only if $b$ is in the set
    \[ \mathcal{B}_{b^*}^{\pm}:= \left\{b\in\RR^d\mid \sgn(\det[\,A_b\mid -c_b\,]_{J})=\sgn(\det[\,A_{b^*}\mid -c_{b^*}\,]_{J}) \text{ for all }J\in\textstyle\binom{[n+1]}{k}\right\}\, . \]
    In particular, $\Trop^+(\langle A_b\, x-c_b\rangle)$ is constant for $b\in \mathcal{B}_{b^*}^{\pm}$. 
    If $b^*\in \mathcal{B}$, then $\mathcal{B}_{b^*}^{\pm}$ is a Euclidean open set.
\end{enumerate}
\end{lemma}

We call the set $\mathcal{B}$ in \Cref{lem:genericMatroid} the \term{generic matroid locus} of the ideal $\langle A x - c \rangle$.

\begin{example}
For the matrix $[\,A \mid -c\,]$ from \Cref{ex:BergmanFanRunning} with a single parameter $b$, all non-vanishing minors  are equal to $\pm 1$, except for 
\[
 \det[\,A \mid -c\,]_{\{1,6\}} = -b\, ,  \qquad    \det[\,A \mid -c\,]_{\{2,6\}}   =  \det[\,A \mid -c\,]_{\{3,6\}} = b   \, .\]
Thus, by \Cref{lem:genericMatroid}, the generic matroid locus is $\mathcal{B}=\CC^*$, whereas
$\mathcal{B}_{b^*}^{\pm}=\RR_{>0}$ for $b^*>0$, $\mathcal{B}_{b^*}^{\pm}=\RR_{<0}$ for $b^*<0$,  
and $\mathcal{B}_{b^*}^{\pm}=\{0\}$ for $b^*=0$.
\end{example}

\subsection{Stable intersection}
\label{subsec:stable_intersection}
In this section, we consider the intersection of tropical varieties. We start by recalling that the \emph{intersection}  of  two polyhedral complexes    $\Sigma_1$ and $\Sigma_2$ in $\RR^n$
  is the polyhedral complex
  \begin{equation*}
\Sigma_1\cap\Sigma_2\coloneqq \{\sigma_1\cap\sigma_2\mid \sigma_1\in\Sigma_1 \text{ and } \sigma_2\in\Sigma_2\}\, .
  \end{equation*}
For $w$ in the support of $\Sigma_1\cap\Sigma_2$, consider cells $\sigma_1\in \Sigma_1$ and $\sigma_2\in \Sigma_2$ such that $w$ belongs to the relative interior of $\sigma_1$ and $\sigma_2$. We say that $\Sigma_1$ and $\Sigma_2$ intersect \emph{transversely} at   $w$ if 
\begin{equation}\label{eq:inter_transverse}
    \dim(\sigma_1+\sigma_2)=n \, .
\end{equation}   
  We say that $\Sigma_1$ and $\Sigma_2$ intersect  \term{transversely} if they intersect 
  transversely at all $w\in \Sigma_1\cap \Sigma_2$. 
Note that the intersection is transversal when it is empty.

 If   $\Sigma_1$ and $\Sigma_2$ are pure  weighted balanced polyhedral complexes that intersect  transversely, then we define the multiplicity of an inclusion-maximal cell $\sigma_1\cap \sigma_2$ of $\Sigma_1\cap\Sigma_2$ as
 \begin{equation}\label{eq:intersection_multiplicities}
    \mult_{\Sigma_1\cap \Sigma_2}(\sigma_1\cap\sigma_2) \coloneqq \mult_{\Sigma_1}(\sigma_1) \mult_{\Sigma_2}(\sigma_2) \, [\ZZ^n:\ZZ^n_{\sigma_1}+\ZZ^n_{\sigma_2}]\,,
  \end{equation}
  where  $\ZZ^n_{\sigma_i}=\ZZ^n\cap \Span(\sigma_i-u_i)$ for any $u_i\in \sigma_i$, and $[\ZZ^n:\ZZ^n_{\sigma_1}+\ZZ^n_{\sigma_2}]$ is the sublattice index. 
 
 Following \cite[Definition 3.6.5, Lemma 3.6.12]{MaclaganSturmfels15}, 
   we define the \term{stable intersection} $\Sigma_1\wedge\Sigma_2$  of 
   two  pure weighted balanced polyhedral complexes $\Sigma_1$ and $\Sigma_2$ in $\RR^n$ as follows:
    \begin{itemize}
        \item  If $\Sigma_1$ and $\Sigma_2$ intersect transversely, then $\Sigma_1\wedge\Sigma_2=\Sigma_1\cap\Sigma_2$ as weighted  polyhedral complexes.

  \item If $\Sigma_1$ and $\Sigma_2$ do not intersect transversely, then
  \begin{equation}\label{eq:shift}
    \Sigma_1\wedge\Sigma_2\coloneqq \lim_{\varepsilon\rightarrow 0} \Sigma_1\wedge (\Sigma_2+\varepsilon\, v)\,,
  \end{equation}
  where $v\in\RR^n$ is any  direction so that $\Sigma_1$ and $\Sigma_2+\varepsilon\, v$ intersect transversely for all $\varepsilon>0$ sufficiently small, and $\lim_{\varepsilon\rightarrow 0}$ also entails adding the multiplicities of all points of $\Sigma_1\wedge (\Sigma_2+\varepsilon\, v)$ with the same limit.
  \end{itemize}
With this definition, $\Sigma_1\wedge\Sigma_2$ is a weighted balanced polyhedral complex. 
By (the proof of) \cite[Proposition 3.6.12]{MaclaganSturmfels15}, there exists a nonempty Zariski open cone in $\RR^n$ of vectors $v$ satisfying the transversality condition required for \eqref{eq:shift}.

Finally, if 
 $\Sigma_1,\Sigma_2$ are pure weighted balanced polyhedral complexes of complementary dimension, we define their \term{intersection number} $\Sigma_1\cdot \Sigma_2$ to be the total multiplicity of $\Sigma_1\wedge\Sigma_2$: 
   \begin{equation}\label{eq:intersection}
    \Sigma_1\cdot \Sigma_2 \coloneqq \deg(\Sigma_1\wedge\Sigma_2)=\sum_{w\in\Sigma_1\wedge\Sigma_2} \mult_{\Sigma_1\wedge\Sigma_2}(w)\, .
  \end{equation}

\begin{example}
\label{ex:stableIntersection}
  Consider $\Trop(\langle x_1\, x_2 -1\rangle)$ from \Cref{ex:polyhedralStructureBinomial} and $\Trop(\langle x_1+x_2-1 \rangle)$ from \Cref{ex:polyhedralStructureAffine}, which do not intersect transversely.  \Cref{fig:stableIntersection} illustrates their stable intersection in two ways using two different translations in \eqref{eq:shift}: $v_1=(0,-1)$ and $v_2=(1,1)$. The intersection number  $\Trop(\langle x_1\, x_2 -1\rangle) \cdot \Trop(\langle x_1+x_2-1 \rangle)=2$ is independent of the translation.
\end{example}

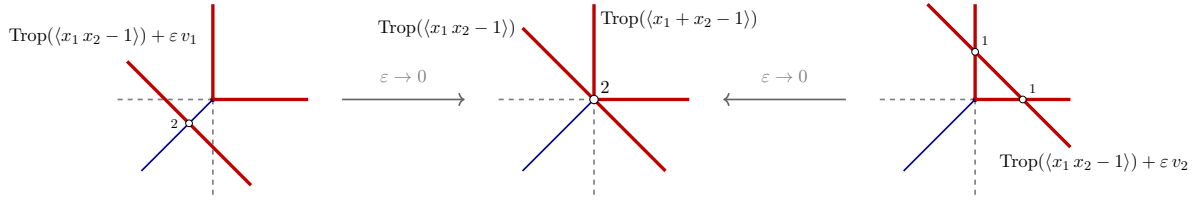
\begin{figure}[t]
  \centering
  \scalebox{0.63}{
  \begin{tikzpicture}
    
    \node (left) at (-8,0)
    {
      \begin{tikzpicture}[every node style/.style={font=\footnotesize}]
        \useasboundingbox (-2,-2) rectangle (2,2);
        \draw[gray,dashed,line width=1pt] (0,-2) -- (0,2);
        \draw[gray,dashed,line width=1pt] (-2,0) -- (2,0);
        \fill(0,0) circle (1.5pt);
        \draw[darkblue,-,line width=1pt] (0,0) -- ++(-1.5,-1.5);
        \draw[darkred,-,line width=2pt] (0,0) -- ++(2,0);
        \draw[darkred,-,line width=2pt] (0,0) -- ++(0,2);

        \node[black,anchor=south,xshift=-5mm] at (-1.8,1) {$\Trop(\langle x_1\, x_2 -1\rangle)+\varepsilon\, v_1$};

        \coordinate (v1) at (0,-1);

        \draw[darkred,line width=2pt]
        (v1) -- ++(0.8,-0.8)
        (v1) -- ++(-1.8,1.8);

        \draw[fill=white] (-0.5,-0.5) circle (2pt);
        \node[anchor=east,xshift=-1mm,font=\scriptsize] at (-0.5,-0.5) {$2$};
      \end{tikzpicture}
    };

    \node (mid) at (0,0)
    {
      \begin{tikzpicture}[every node style/.style={font=\footnotesize}]
        \useasboundingbox (-2,-2) rectangle (2,2);
        \draw[gray,dashed,line width=1pt] (0,-2) -- (0,2);
        \draw[gray,dashed,line width=1pt] (-2,0) -- (2,0);

        \draw[darkblue,-,line width=1pt] (0,0) -- ++(-1.5,-1.5);
        \draw[darkred,-,line width=2pt] (0,0) -- ++(2,0);
        \draw[darkred,-,line width=2pt] (0,0) -- ++(0,2);

        \draw[darkred,line width=2pt]
        (0,0) -- ++(1.5,-1.5)
        (0,0) -- ++(-1.5,1.5);

        \node[black,anchor=north west] at (0,2) {$\Trop(\langle x_1+x_2-1\rangle)$};
        \node[black,anchor=east] at (-1.5,1.5) {$\Trop(\langle x_1\, x_2 -1\rangle)$};

        \draw[fill=white] (0,0) circle (2.5pt);
        \node[anchor=south west] at (0,0) {$2$};
      \end{tikzpicture}
    };

    \node (right) at (8,0)
    {
      \begin{tikzpicture}[every node style/.style={font=\footnotesize}]
        \useasboundingbox (-2,-2) rectangle (2,2);
        \draw[gray,dashed,line width=1pt] (0,-2) -- (0,2);
        \draw[gray,dashed,line width=1pt] (-2,0) -- (2,0);
        \fill(0,0) circle (1.5pt);
        \draw[darkblue,-,line width=1pt] (0,0) -- ++(-1.5,-1.5);
        \draw[darkred,-,line width=2pt] (0,0) -- ++(2,0);
        \draw[darkred,-,line width=2pt] (0,0) -- ++(0,2);

        \coordinate (v2) at (0.5,0.5);
        \draw[darkred,line width=2pt]
        (v2) -- ++(1.5,-1.5)
        (v2) -- ++(-1.5,1.5);

        \node[black,anchor=north,xshift=5mm] at (2,-1) {$\Trop(\langle x_1\, x_2-1\rangle)+\varepsilon \, v_2$};

        \draw[fill=white]
        (1,0) circle (2pt)
        (0,1) circle (2pt);

        \node[anchor=south west,font=\scriptsize] at (1,0) {$1$};
        \node[anchor=south west,font=\scriptsize] at (0,1) {$1$};
      \end{tikzpicture}
    };

    \draw[->,line width=1pt,black!60,shorten >=1.5em,shorten <=1.5em] (left.east) -- (mid.west);
    \draw[->,line width=1pt,black!60,shorten >=1.5em,shorten <=1.5em] (right.west) -- (mid.east);

    \node[gray] at ($(left.east)!0.5!(mid.west) + (0,+0.5)$) {${\small \varepsilon\rightarrow 0}$};
    \node[gray] at ($(right.west)!0.5!(mid.east) + (0,+0.5)$) {${\small \varepsilon\rightarrow 0}$};
  \end{tikzpicture}
  }

  \caption{The stable intersection $\Trop(\langle x_1\, x_2 -1\rangle)\wedge\Trop(\langle x_1+x_2-1\rangle)$ from \Cref{ex:stableIntersection} computed separately with the translations $v_1=(0,-1)$ and $v_2=(1,1)$. The label of each intersection point is its multiplicity. For the translation vector $v_1$, the sublattice index is $2$. }
  \label{fig:stableIntersection}
\end{figure}

The following results from \cite{BogartJensenSpeyerSturmfelsThomas07,OssermanPayne13} allow us to compute tropicalizations for ideals that are decomposed into simpler summands, provided that the tropicalizations of each of the summands intersect transversely.

\begin{theorem}[Transverse intersection theorem]
\label{thm:transverse_intersection_theorem}
    Let $I,J\subseteq \mathbb{K}[x^\pm]$ be pure-dimensional complete intersection ideals, and suppose that $\Trop(I)$ and $\Trop(J)$ intersect transversely.  Then we have the following equality of sets:
    \begin{align}
        \label{eq:TransversIntersection_sets}
    \Trop(I+J)=\Trop(I)\cap \Trop(J).
    \end{align}
    Moreover, for every point $w$ in~\eqref{eq:TransversIntersection_sets}, the cells of the weighted polyhedral complexes $\Trop(I+J)$ and $\Trop(I) \cap \Trop(J)$ whose relative interiors contain $w$ have the same multiplicities. 
\end{theorem}

An analog of the transverse intersection theorem for positive tropical varieties was developed in \cite[Theorem 4.11]{RoseTelek24}, motivated by the goal of deriving lower bounds on the maximal positive real roots of purely vertically parametrized systems. Here, we prove a version of \cite[Corollary 5.5]{RoseTelek24} geared towards computing lower bounds of the maximal number of positive \emph{nondegenerate} zeros of a parametric polynomial system.
For this purpose, we consider the subfield $\RR\{t\}$ of   $\RRt$ consisting of real \emph{algebraic} Puiseux series, as well as the multiplicative subgroup $\RR\{t\}_{>0} =\RRt_{>0} \cap \RR\{t \}$. The field  $\mathbb{R}\{t\}$ is  real closed and contains the real Puiseux series that are algebraic over the rational function field $\mathbb{R}(t)$.

The elements of $\mathbb{R}\{t\}$ represent semi-algebraic continuous functions to the right of the origin \cite[Theorem~3.14]{BasuPollackRoy}, so each $z \in \mathbb{R}\{t\}$ defines a real number $z(t) \in \mathbb{R}$ for all $t>0$ sufficiently small.
Thus, every  polynomial $f \in \mathbb{R}\{t\}[x^\pm]$ gives rise to a polynomial with real coefficients $f_t \in \mathbb{R}[x^\pm]$ when evaluated at a small value $t>0$. We extend this construction to ideals and write $I_t \coloneqq \langle f_{1,t}, \dots , f_{k,t} \rangle \subseteq \mathbb{R}[x^\pm]$ for the ideal generated by the evaluations of the generators of $I = \langle f_{1}, \dots , f_{k} \rangle \subseteq  \mathbb{R}\{t\}[x^\pm]$.

\begin{theorem}
\label{thm:RealTropLowerBounds}
   Let $I,J \subseteq  \mathbb{R}\{t\}[x^\pm]$ be prime ideals such that $V(I)$ and $V(J)$ have complementary dimension, $V(I+J)\cap \RRt^n_{>0}$ is finite, and all zeros of $I+J$  in $\RR\{t\}^n_{>0}$ are nondegenerate. Suppose that $\Trop(I)$ and $\Trop(J)$ intersect transversely and that all intersection points belong to maximal polyhedra of multiplicity one of both $\Trop(I)$ and $\Trop(J)$. Then, for all $t>0$ sufficiently small,  all zeros of $I_t+ J_t$ in $\mathbb{R}^n_{>0}$ are nondegenerate and  
   \begin{equation}\label{eq:trop_plus}
   \#  V_{>0}(I_t+ J_t)  \geq \#\left(\Trop^+(I) \cap \Trop^+(J) \right)\,.
    \end{equation}
  The inequality \eqref{eq:trop_plus} holds with equality if the right-hand side equals zero.
\end{theorem}
\begin{proof}
The statement is a consequence of the following equalities and inequalities: 
\begin{align*} 
& \# \big(\Trop^+(I) \cap \Trop^+(J)\big)  \overset{(1)}{=} \#  \Trop^+(I+J) 
\overset{(2)}{=} \# \val\big( V(I+J) \cap \RRt^n_{>0}\big) \\ &\qquad  
\overset{(3)}{\leq}  \# \big(V(I+J) \cap \RRt^n_{>0}\big)\, 
 \overset{(4)}{=} \# ( V(I+J) \cap \mathbb{R}\{t\}^n_{>0}) 
 \overset{(5)}{=} \# V_{>0}(I_t+J_t) \quad \text{for $t>0$ small}\, . 
\end{align*}
 Specifically: (1) follows from \cite[Theorem~4.11]{RoseTelek24}; (2) holds as $V(I+J)\cap \RRt^n_{>0}$ is finite; (3) is immediate; (4) follows from the Tarski--Seidenberg Principle (also called the Transfer Principle for real closed fields) \cite[Theorem~2.80]{BasuPollackRoy}, as both $\RRt$ and $\RR\{t\}$ are real closed fields and the cardinality of a finite set can be formulated as a first-order sentence in the language of ordered fields. 
(5) follows from \cite[Proposition~3.17]{BasuPollackRoy}, as  all zeros  of  $I+J$ in $\mathbb{R}\{t\}^n_{>0}$ are nondegenerate, and  nondegeneracy can be formulated as a first-order sentence. 
Finally, the last statement in the theorem follows by realizing that (3) holds with equality if $\Trop^+(I) \cap \Trop^+(J)=\emptyset$.
\end{proof}

\begin{example}
    Consider the ideals $I = \langle x_1+x_2-1\rangle$ and $J^h = \langle x_1\, x_2 - t^h \rangle$ in $\mathbb{R}\{t\}[x_1^\pm,x_2^\pm]$ for $h\in \QQ$. A simple computation shows that
    \[ V(I+J^h) = \left\{ \Big(\tfrac{1 + \sqrt{1-4t^h} }{2},\tfrac{1 - \sqrt{1-4t^h} }{2}\Big),\Big(\tfrac{1 - \sqrt{1-4t^h} }{2},\tfrac{1 + \sqrt{1-4t^h} }{2}\Big) \right\}.\] 
    Thus, for $h = 1$ and $t > 0$ sufficiently small, $V(I_t+J^h_t)$ has two points in $\mathbb{R}_{>0}^2$,  and for $h = -1$ and $t > 0$ sufficiently small, we have $V_{>0}(I_t+J^{h}_t) = \emptyset$.

Let us recover the same result with tropical methods by applying \Cref{thm:RealTropLowerBounds}. 
In \Cref{ex:polyhedralStructureAffine}, we computed $\Trop^+(\langle x_1+x_2-1\rangle)$, while  
\Cref{lem:binomial} tells us that 
\[ \Trop^+( \langle x_1\, x_2 - t^h \rangle)=\Trop(\langle x_1\, x_2 - t^h \rangle) =\{w\in \RR^2 \mid w_1+w_2=h\}\,.\]
The intersection $\Trop^+(\langle x_1+x_2-1\rangle) \cap \Trop^+( \langle x_1\, x_2 - t^h\rangle)$ for $h=1$ consists of two points (right-hand side of \Cref{fig:stableIntersection}), and is empty for $h=-1$ (left-hand side of \Cref{fig:stableIntersection}). 
Hence, by \Cref{thm:RealTropLowerBounds},  for $t>0$ small enough, $V_{>0}(I_t+J^h_t)$ has two points if $h=1$ and none if $h=-1$. 

This example illustrates that, unlike the stable intersection of tropical varieties, the number of intersection points of positive tropical varieties may depend on the translation vector.
\end{example}

We conclude this subsection with a technical lemma capturing the key step in the proof of \Cref{thm:monomial_reembedding_bounds,thm:bound_purevertical}.

\begin{lemma}\label{lem:technical}
Let $I,J\subseteq \CC[p][x^\pm]$ be prime  complete intersection ideals for parameters $p=(p_1,\dots,p_k)$ 
such that $V(I)$ and $V(J)$ have complementary dimension and $I+J$ has generically nondegenerate zeros. Let $\Trop(I)$ and $\Trop(J)$ be their tropical varieties for generic complex parameter values. 
\begin{enumerate}[label=(\roman*)]
    \item Assume that there exists $\hat{h}\in \QQ^n$ such that $\Trop(I)+\hat{h}$ and $\Trop(J)$ intersect transversely,  and $\Trop(I)+\hat{h}=\Trop(I_p)$ and $\Trop(J)=\Trop(J_p)$ for all $p$ in a  Zariski dense subset $\mathcal{P}$ of $\mathbb{K}^k$.
Then
\begin{equation}\label{eq:technical}
\grc(I+J)= \deg\big((\Trop(I)+\hat{h}) \cap \Trop(J)\big) \, .
\end{equation}
\item Assume that there exists $h\in \QQ^n$ such that for all $\varepsilon>0$ small enough and
$\hat{h} = \varepsilon \, h$,  
$\Trop(I)+\hat{h}$ and $\Trop(J)$ intersect transversely and \eqref{eq:technical} holds. Then
\[ \grc(I+J) = \Trop(I) \cdot \Trop(J)\, . \]
\end{enumerate}
\end{lemma}
\begin{proof}
(i) By denseness of $\mathcal{P}$, we can choose $p\in \mathcal{P}$ such that all zeros of $I_p+J_p$ are nondegenerate. Then it holds that
\begin{align*}
\grc(I+J)=\#V(I_p+J_p)=\deg(\Trop(I_p+J_p))=\deg((\Trop(I)+\hat{h})\cap \Trop(J))\,,
\end{align*}
where the
second equality follows from \Cref{thm:zero-dimensional_case} and the third follows from the transverse intersection theorem (\Cref{thm:transverse_intersection_theorem}).

(ii) Take the limit as $\varepsilon\rightarrow 0$ of both sides of \eqref{eq:technical} and use \eqref{eq:shift} and \eqref{eq:intersection}.
\end{proof}

\subsection{Generic and maximal root counts}
\label{subsec:root_bounds}
We are now equipped with the tools  to 
study the root counts $\grc(F)$ and $\mrc_{>0}(F)$ for an augmented vertically parametrized system using its toric reembedding  \eqref{eq:monomial_reembedding}.
The key ingredients are the tropicalizations of the linear and binomial parts of the reembedding together with \Cref{thm:RealTropLowerBounds} and \Cref{lem:technical}. 

Consider a square augmented vertically parametrized system 
\[ F=\big(\,C  x^M ,\: Lx-b\,\big)\in \CC[a,b][x^\pm]^n\]
with  vertical matrix $C\in \CC[a]^{s\times r}$, matrix $L\in\CC^{d\times n}$ of full row rank, and exponent matrix $M\in \mathbb{Z}^{n\times r}$. 
We obtain after the toric reembedding from \eqref{eq:monomial_reembedding} 
the following ideals in \eqref{eq:monomial_linbin}:
\begin{equation*} 
I^{\bin}\coloneqq \langle y - x^{M}  \rangle \subseteq \CC[y^\pm,x^\pm]\,, \qquad  I^{\lin} \coloneqq \langle C y  , Lx - b  \rangle \subseteq \CC[a,b][y^\pm,x^\pm]\, . 
\end{equation*}
The following holds for their tropicalizations as discussed previously:
\begin{itemize}
    \item Since $I^{\bin}$ is prime, \Cref{lem:binomial} tells us that 
    \[ \Trop(I^{\bin})=\Trop^+(I^{\bin})=\ker\,  [\,\id_{r} \mid -M^\top\,] = \rowspan\,[\,M\mid \id_n\,] \, . \] 
    \item For each $a\in \CC^m$ and $b\in\CC^d$,  $\Trop(I^{\lin}_{a,b})$ is determined by the matroid of
    \begin{equation}
    \label{eq:block_matrix_for_linear_part}
    \begin{bmatrix}C_a & 0 & 0 \\ 0 & L &-b\end{bmatrix}\in\CC^{n\times (r+n+1)}\, . 
    \end{equation}
           By \Cref{lem:genericMatroid}, there is a nonempty Zariski open subset $\mathcal{B} \subseteq \CC^{m+d}$ such that $\Trop(I_{a,b}^{\lin})$ is constant for $(a,b)\in\mathcal{B}$.
           We denote this generic tropicalized affine space by $\Trop(I^{\lin})$. 
\end{itemize}

\smallskip
We introduce the partition 
$[m] = J_1 \sqcup \dots \sqcup J_r$, such that for each $k\in [r]$,  
\begin{equation*}\label{eq:partition}
i\in J_k \quad \Leftrightarrow\quad \text{$a_i$ appears in the $k$-th column of $C$}\, ,
\end{equation*}
and consider the embedding
\begin{equation*}\label{eq:nu}
\eta \colon \mathbb{K}^r \hookrightarrow \mathbb{K}^m, \qquad \eta(z)_j= z_k \quad \text{for all} \ j \in J_k \,.
\end{equation*}

\begin{lemma}\label{lem:nu}
    With the notation above, let $(a,b)\in \mathbb{K}^m \times \mathbb{K}^d$ and $z\in \mathbb{K}^r$. Then
    \begin{align*}
    \Trop(I^{\lin}_{a\star \eta(z), b}) & = \Trop(I^{\lin}_{a,b}) + (-\val(z),\mathbf{0}_n)  \\ 
    \Trop^+(I^{\lin}_{a\star \eta(z), b}) & = \Trop^+(I^{\lin}_{a,b}) + (-\val(z),\mathbf{0}_n) \, \qquad  \text{if }z\in \RRt^r_{>0}\,.
    \end{align*}
\end{lemma}
\begin{proof}
By definition of the embedding $\eta$, 
we have $C_{a\star \eta(z)} =C_{a} \diag(z)$ and hence
\[ (y,x)\in V\big(I^{\lin}_{a\star \eta(z),b}\big) \quad \Leftrightarrow \quad (z \star y,x)\in V(I^{\lin}_{a,b})\,. \]
The two equalities follow now from \eqref{eq:tropicalization} and \eqref{eq:trop_pos}. 
  \end{proof}

With this in place, we can state the main result of this section. Part (i) formalizes \cite[Example 6.7]{HelminckRen22}, thereby generalizing \cite[Proposition 6.5]{HelminckRen22} to \emph{augmented} vertically parametrized systems. Similarly, part (ii) generalizes \cite[Theorem 6.1]{RoseTelek24} to augmented vertically parametrized systems and  strengthens it by providing lower bounds on the number of \emph{nondegenerate} positive real solutions.  Additionally, our proof is more elementary, using only basic theory of algebraic and tropical varieties.

\begin{theorem}[Complex and positive root bounds]
\label{thm:monomial_reembedding_bounds}
Let $F=(C x^M ,\: L x - b)\, \in \CC[a,b][x^\pm]^n$ be a square augmented vertically parametrized system
with   $n=s+d$, $C\in \CC[a]^{s\times r}$ vertical, 
$L\in \mathbb{C}^{d\times n}$ of full row rank, and  $M\in \mathbb{Z}^{n\times r}$. 
Let $I^{\lin}$ and $I^{\bin}$ be as in \eqref{eq:monomial_linbin} and $\mathcal{B}\subseteq \CC^m\times\CC^d$ be the  generic matroid locus of $I^{\lin}$.  

Suppose $h\in\QQ^r$ is such that
$\Trop(I^{\bin})+(h,\mathbf{0}_n)$ and $\Trop(I^{\lin})$ intersect transversely. 
For  $(a^*,b^*)\in\mathcal{B}$, the following statements hold:
\begin{enumerate}[label=(\roman*)]
    \item We have
    \[\grc(F)=\deg\big((\Trop(I^{\bin})+(h,\mathbf{0}_{n} )) \cap \Trop(I^{\lin}_{a^*,b^*})\big)\, .\]

   \item 
  If the coefficients of $F$ are real
    and 
    $(a^*,b^*)\in\RR^m_{>0}\times L(\RR^n_{>0})$, then 
    \begin{align}
\label{eq:num_pos_points1} 
    \mrc_{>0}(F)\geq\#\big( (\Trop^+(I^{\bin})+(h,\mathbf{0}_{n} ))\cap \Trop^+(I_{a^*,b^*}^{\lin})\,\big) \, . 
    \end{align}
   Additionally, the right-hand side of \eqref{eq:num_pos_points1} is a lower bound for the number of nondegenerate positive zeros of  $F_{a\star \eta(t^h),b}$ for all $t>0$ sufficiently small, and generic parameters $(a,b)$ in a  neighbourhood of $(a^*,b^*)$.  
\end{enumerate}
\end{theorem}

\begin{proof}
 Recall  from  \eqref{eq:monomial_correspondence} that 
$\grc(F)=\grc(I^{\bin}+I^{\lin})$, and hence it is enough to focus on the latter.  Let $\mathcal{V}\subseteq \mathbb{K}^m\times \mathbb{K}^d$ be   a nonempty Zariski open subset  of parameters 
$(a,b)$ such that $\grc(F)=\#V(F_{a,b})$ (which also implies that all zeros of $F_{a,b}$ are nondegenerate).
Consider the automorphism $\mu$ of $\mathbb{K}^m\times \mathbb{K}^d$ defined by $\mu(a,b)=(a\star \eta(t^h),b)$. 
Using \Cref{lem:nu} with $z=t^h$, 
we have
\[\Trop(I^{\lin}_{\mu(a,b)}) = \Trop(I^{\lin}) + (-h,\mathbf{0}_n)\quad\text{for all }(a,b) \in \mathcal{B}\, .\]
Since $\mu$ is an automorphism, $\mu( \mathcal{B})$ is Zariski dense in $\mathbb{K}^m\times \mathbb{K}^d$. By letting $\mathcal{P}=\mu( \mathcal{B})$, $\hat{h}=(-h,\mathbf{0}_n)$, $J=I^{\bin}$, and $I=I^{\lin}$ in \Cref{lem:technical}(i), we obtain
\begin{align*}
    \grc(I^{\bin}+I^{\lin}) & =  \deg(\Trop(I^{\bin}) \cap (\Trop(  I^{\lin})+(-h,\mathbf{0}_n)))
  \\
    &= \deg\big((\Trop(I^{\bin})+(h,\mathbf{0}_{n} )) \cap \Trop(I^{\lin}_{a^*,b^*})\big)\, .
\end{align*} 
This shows  part  (i).

For part (ii), as $L$ has full row rank, using \Cref{lem:genericMatroid} we consider the nonempty Euclidean open subset  
\[\mathcal{B}'\coloneqq\mathcal{B}_{(a^*,b^*)}^\pm \cap (\RR^{m}_{>0} \times L(\RR_{>0}^{n})) \subseteq \RR^{m}_{>0} \times L(\RR_{>0}^{n})\] 
  such that $\Trop^+(I^{\lin}_{a,b}) = \Trop^+(I^{\lin}_{a^*,b^*})$ if $(a,b)\in \mathcal{B}'$. 
 In particular, 
  $\mu(\mathcal{B}') \cap \mathcal{V}\neq \emptyset$ and $\mathcal{B}' \cap \mu^{-1}(\mathcal{V})$ is a nonempty Euclidean open subset of $\RR^{m}_{>0} \times L(\RR_{>0}^{n})$. For every
 $(a,b) \in \mu(\mathcal{B}') \cap \mathcal{V}$,
\Cref{lem:nu}   gives that $\Trop^+(I^{\lin}_{a,b}) = \Trop^+(I^{\lin}_{a^*,b^*})  + (-h,\mathbf{0}_n)$. Part (ii) follows by applying \Cref{thm:RealTropLowerBounds}  to the ideals $I^{\lin}_{a,b}$ and $I^{\bin}$ of $\RR\{t\}[y^\pm,x^\pm]$. \end{proof} 

\begin{remark}\label{rem:unbounded_regions}
The lower bound \eqref{eq:num_pos_points1} depends on the choices of $a^*$, $b^*$ and $h$, and an interesting question for future work is how to make these choices systematically to obtain the sharpest possible bound on $\mrc_{>0}(F)$. Likewise, one can ask in what cases it is possible for the bound in \eqref{eq:num_pos_points1} to be equal to $\mrc_{>0}(F)$. The last sentence in part (ii) of \Cref{thm:monomial_reembedding_bounds} gives an immediate obstruction in this regard, since \eqref{eq:num_pos_points1} many positive nondegenerate zeros must be attainable for ``extreme'' parameter values $a\star \eta(t^h)$, obtained by letting $t\to 0$.

As a simple example of this, consider the vertical system 
$F=(a_1 x_1 - a_2  + a_3,  a_4 x_2 + a_2- 2a_3)$. For $a\in \RR^4_{>0}$, the system clearly has at most one positive zero, and it has exactly one if and only if $2a_3>a_2>a_3$.
However, the right-hand side of \eqref{eq:num_pos_points1} will always be zero. To see this, assume the right-hand side is equal to one for some $a^*\in \RR^4_{>0}$ and $h\in \RR^4$.
As \eqref{eq:num_pos_points1} is invariant under small  perturbations of $h$, we can take $h$ with $h_2\neq h_3$. 
Then for some $a\in\RR^4_{>0}$ and all $t>0$ small enough, the system $F_{a\star t^h}$ will have one positive zero, meaning that $2a_3t^{h_3} >a_2t^{h_2} >a_3t^{h_3}$ holds for all $t>0$ arbitrarily small, which is a contradiction. 
\end{remark} 

From \Cref{thm:monomial_reembedding_bounds}, we obtain \Cref{alg:bounds} for computing the generic root count $\grc(F)$ and a lower bound for the maximal positive root count $\mrc_{>0}(F)$. 
An implementation is available as part of our Julia package \texttt{VerticalRootCounts.jl}, and its computational performance is showcased in \Cref{sec:implementation}.

\begin{algorithm}[h]
\vspace{0.2em}
\caption{Root bounds of augmented vertically parametrized systems}
\label{alg:bounds}
\KwIn{A square augmented vertically parametrized system $F=(C  x^M ,Lx-b)$.} 
\KwOut{The generic root count $\grc(F)$ and a lower bound of $\mrc_{>0}(F)$.} 
\vspace{0.2em}
Pick at random $a\in \QQ^m_{>0}, b\in L(\QQ_{>0}^n)$, and certify that evaluating the matrix of \eqref{eq:block_matrix_for_linear_part} at the point $(a,b)$ gives rise to the generic matroid via \Cref{lem:genericMatroid}.\\[0.2em]
Compute $\Trop(I^{\bin})$ and $\Trop(I^{\lin}_{a,b})$.\\[0.2em]
Pick at random $h\in\QQ^r$, and certify that $\Trop(I^{\bin})+(h,\mathbf{0}_n)$ and $\Trop(I^{\lin}_{a,b})$ intersect transversely via \eqref{eq:inter_transverse}.\\[0.2em]
Compute $\Sigma:=(\Trop(I^{\bin})+(h,\mathbf{0}_n))\cap \Trop(I^{\lin}_{a,b})$.\\[0.2em]
Count the number of points of $\Sigma$ that belong to $\Trop^+(I^{\lin}_{a,b})$. Return the number as a lower bound of $\mrc_{>0}(F)$, and return $\deg(\Sigma)$ as $\grc(F)$.\\[0.2em]
\end{algorithm}

Next, we argue that the vector $h$  in the third line in \Cref{alg:bounds} almost surely leads to transversely intersecting tropical varieties. 
By \cite[Proposition~3.6.12]{MaclaganSturmfels15}, for a generic translation~$\hat{h}$, the two tropical varieties will intersect transversely. Here we are claiming it is enough to consider translations of the form $\hat{h}=(h,\mathbf{0}_n)$.  
The following result is a slight strengthening of \cite[Proposition~3.6.12]{MaclaganSturmfels15} adjusted to our setting.

\begin{lemma}[Generic transversality]\label{lem:generic_transversality}
Let $\Sigma$ be a polyhedral fan in $\RR^{r+n}$  pure of dimension~$r$ and let $M\in \mathbb{Z}^{n\times r}$.  Then there exists a nonempty Zariski open cone $\mathcal{H} \subseteq \mathbb{R}^r$ 
such that the intersection $\Sigma \cap (\rowspan  [\,M\mid \id_n\,] +(h,\mathbf{0}_{n} ) )$ is transverse for all $h\in\mathcal{H}$.
\end{lemma}

\begin{proof}
For brevity, we write 
$V \coloneqq \rowspan  [\,M\mid \id_n\,]\subseteq \RR^{r+n}$.
As the dimension of $V$ is $n$, the fan $\Sigma$ and the affine space $V +(h,\mathbf{0}_{n} )$ are of complementary dimension for all vectors~$h$. 
 We now follow the proof of \cite[Proposition~3.6.12]{MaclaganSturmfels15}. 
For every cone $\tau \in \Sigma$  such that 
$\dim (\tau + V) <r+n$, choose a nonzero vector $u_\tau \in (\Span(\tau)+ V)^\perp$.

If $\Sigma$ and $V+(h,\mathbf{0}_{n} )$ do not intersect transversely at a point $w \in \RR^{r+n}$, then there exists a cone 
$\tau \in \Sigma$ with $w \in \relint(\tau)$ such that 
$\dim (\tau+V) <r+n$.  Write $w \in V +(h,\mathbf{0}_{n})$ as $w = v + (h,\mathbf{0}_{n} )$ with $v\in V$. Using  that $u_\tau \cdot w = 0$ and $u_\tau \cdot v = 0$, we have $u_\tau \cdot (h,\mathbf{0}_{n} ) = 0$.
By letting $\overline{u}_\tau$ be the projection of $u_\tau$ onto the first $r$ coordinates,  it follows that
\begin{align*}
    \mathcal{H} \coloneqq \RR^r \setminus 
   \bigcup_{\tau \mid \dim (\tau + V) <n +r}  \Span(\overline{u}_\tau) ^\perp
\end{align*}
is Zariski open by construction, a cone, and for every $h \in \mathcal{H}$, the fan $\Sigma$ and the affine space $V+(h,\mathbf{0}_{n} )$ intersect transversely.  
By definition of the space $V$, using that  $u_\tau$ is orthogonal to $V$ we conclude that 
$\overline{u}_\tau\neq 0$ for all $\tau$ and hence $\mathcal{H}\neq \emptyset$. 
This concludes the proof. 
\end{proof}

\begin{theorem}[Generic root count]\label{thm:grc}
Let $F$ be a square augmented vertically parametrized system in $\CC[a,b][x^\pm]^n$ and  
let $I^{\lin}$ and $I^{\bin}$ be as defined in \eqref{eq:monomial_linbin}. Then
\[\grc(F)= \Trop(I^{\bin}) \cdot \Trop(I^{\lin})\, .\]
\end{theorem}
\begin{proof}
Let $h\in \QQ^r$ belong to the cone $\mathcal{H}$ of \Cref{lem:generic_transversality}. Then,
$\hat{h}=\varepsilon\, (h, \mathbf{0}_{n})$ satisfies the hypotheses of \Cref{lem:technical}(ii).
\end{proof}

In the following corollary,
the degree of a polynomial ideal refers to the number of intersection points (counted with multiplicity) of its variety with a generic linear space of complementary dimension.

\begin{corollary}[Generic degree]\label{cor:generic_degree}
Let $C x^M  \, \in \CC[a][x^\pm]^s$ be a  vertically parametrized system
with    $C\in \CC[a]^{s\times r}$ vertical, and  $M\in \mathbb{Z}^{n\times r}$. For $L\in \CC^{(n-s)\times n}$ such that $\mathcal{M}^*[L]$ is uniform, consider the augmented vertically parametrized system $F=(C x^M,Lx-b)$ and its associated ideals $I^{\lin}$ and $I^{\bin}$ as in \eqref{eq:monomial_linbin}. Then 
 the generic degree of the ideal 
 $\langle C_a x^M \rangle$  is
\begin{equation}\label{eq:generic_degree} 
\Trop(I^{\bin}) \cdot \Trop(I^{\lin})\, .
\end{equation}
\end{corollary}
\begin{proof}
Let $G=C x^M$. 
By \cite[Theorem 3.7]{FeliuHenrikssonPascual2025vertical}, 
the varieties $V_{\CC^*}(G_a)$ are either generically empty or generically of dimension $n-s$. In the first case, the degree is generically zero and the intersection \eqref{eq:generic_degree} is zero. In the second case, 
note that for any fixed $L'\in\CC^{(n-s)\times n}$ of full row rank, the zeros of $\langle G_a,L' x-b \rangle$ are  nondegenerate for generic $(a,b)$. Hence, the generic degree of $G$ agrees with the maximal value of
$\grc(G,L'x-b)$ over all $L'$.
As $\mathcal{M}^*[L]$ is uniform, it follows from \Cref{lem:genericMatroid} that $\Trop(\langle Lx-b\rangle)$ contains $\Trop(\langle L'x-b\rangle)$ for any $L'$ and $b\in\CC^{n-s}$. The statement now follows from \Cref{thm:grc}. 
\end{proof}

For the special situation of a square vertically parametrized system, i.e., $n=s$ polynomials and variables, and no augmenting linear forms, i.e., $d=0$, the computation of the number $\grc(F)$ is easier as, in this case, one can eliminate the $x$ variables, and thus reduce the problem to a stable intersection in $\RR^r$ rather than $\RR^{r+n}$. 
We restrict to the scenario where $M$ has full row rank, as when this is not the case, the generic root count is zero. Recall that we defined the monomial map $\phi_M$ in \Cref{subsec:preliminaries} above. 

\begin{theorem}[Purely vertically parametrized systems]
\label{thm:bound_purevertical}
Let $C \in \mathbb{C}[a]^{n \times r}$ be   a vertical matrix, $a=(a_1,\dots,a_m)$, and $M \in \mathbb{Z}^{n \times r}$ be of rank $n$. Consider the ideals $I^{\lin} \coloneqq \langle\,C y\,\rangle$, with generic matroid locus $\mathcal{B}\subseteq\CC^m$, and $I_M\coloneqq I(\im(\phi_M))$ in $\CC[a][y^\pm]$ in the variables $y=(y_1,\dots,y_r)$. 
\begin{enumerate}[label=(\roman*)]
    \item  It holds that 
\begin{align*}
\grc ( I^{\lin} + I_M  ) &= \rowspan(M)\cdot \Trop(I^{\lin} )\, \\
\grc(C x^M  ) &=  \deg(\phi_M) \, \big(\rowspan(M)\cdot \Trop(I^{\lin})\big)\, .
\end{align*}
\item If $C$ has real coefficients, then 
\begin{align}
\label{eq:max_num_pos_points} \mrc_{>0}(C  x^M ) \leq  \rowspan(M)\cdot \Trop(I^{\lin})  \,, 
\end{align}
and  for  any $a^*\in \mathcal{B}\cap\RR^m_{>0}$
and $h\in\QQ^r$ such that 
$\rowspan(M)+h$ and $\Trop(I^{\lin}_{a^*})$ intersect transversely, it holds that
\begin{align}
\label{eq:num_pos_points} 
\mrc_{>0}(C x^M) \geq \# \big((\rowspan(M)+h)\cap \Trop^+(I^{\lin}_{a^*})\big)\, .
\end{align}
Additionally, the right-hand side of \eqref{eq:num_pos_points} is a lower bound for the number of nondegenerate positive zeros of $C_{a\star \eta(t^h)} x^M$ if $t>0$ is sufficiently small and for generic $a$ in a neighborhood of $a^*$. 
\end{enumerate}
\end{theorem}

\begin{proof} 
Let $F=C x^M$. 
Extending $\phi_M$ to $\mathbb{K}^n$,  we have that
$V( F_a )=\phi_M^{-1}\big( V(C_a y) \cap \im(\phi_M)\big)$ for all  
$a\in \mathbb{K}^m$, 
and thus, $\grc ( F) = \deg(\phi_M)  \grc\big(I^{\lin} + I_M \big)$.
An easy computation gives
\[\Trop(I_M)= \Trop^+(I_M)= \rowspan(M)\,, \]
with a single polyhedron of multiplicity 1. 
Observe that 
$I^{\lin}_a + I_M$ is a complete intersection ideal, and furthermore, 
if $x$ is a nondegenerate zero of $F_a$, then $x^M$ is nondegenerate as a zero of 
$I^{\lin}_a + I_M$. Therefore, 
as $F$ has generically nondegenerate zeros, so does $I^{\lin} + I_M$. 

Let $h\in \QQ^r$ be such that $\Trop(I^{\lin})$ and $\rowspan(M)+ \varepsilon\, h$
intersect transversely  for all $\varepsilon>0$ (see discussion after \eqref{eq:shift}). 
Similar to the proof of \Cref{thm:monomial_reembedding_bounds}, 
consider the automorphism $\mu$ of $\mathbb{K}^m$ defined by $\mu(a)=a\star \eta(t^h)$. Then, as $\mu(\mathcal{B})   \neq \emptyset$ is Zariski dense,  
using \Cref{lem:nu}  and  \Cref{lem:technical}(ii) with the shift $h$, and the ideals $I=I^{\lin}$ and $J=I_M$, we obtain
\begin{equation}\label{eq:toric}
\grc (I^{\lin}+ I_M) =  \rowspan(M)\cdot \Trop(I^{\lin})\, .
\end{equation}
This proves (i). 
To show \eqref{eq:max_num_pos_points} in part (ii), we use that the coefficients are real and that $\phi_M$ is injective when restricted to $\RR^n_{>0}$. Then  for all fixed $a^*\in \RR^m_{>0}$: 
\begin{align*}
&\#\big \{x\in \RR^n_{>0} \mid x  \text{ is a nondegenerate zero of }F_{a^*} \big\}  \\
&\hspace{2cm} =\# \{y\in \RR^r_{>0} \mid y  \text{ is a nondegenerate zero of } I_{a^*}^{\lin} + I_M  \} \\ 
&\hspace{2cm}\leq  \max_{a\in \CC^m} \ \# \{y \in (\CC^*)^r \mid y  \text{ is a nondegenerate zero of } I^{\lin}_{a} +I_M  \} \\
&\hspace{2cm}=\grc ( I^{\lin} + I_M)=\rowspan(M)\cdot \Trop(I^{\lin}) 
\end{align*}
by \eqref{eq:toric}. Equation \eqref{eq:max_num_pos_points}  thus holds.  
The second part of (ii) is simply  \Cref{thm:monomial_reembedding_bounds}(ii) 
 applied to the case where $L$ is empty, 
 as 
\[ \# (\rowspan\,[\,M\mid \id_n\,]+(h,\mathbf{0}_{n} ))\cap \Trop^+(\langle C_a y  \rangle) =\# (\rowspan(M)+h) \cap \Trop^+(\langle C_a y   \rangle)  \, , \]
where the set on the left-hand side belongs to $\RR^{r+n}$. 
\end{proof}

\begin{remark}
Alternatively, \Cref{thm:bound_purevertical}(i) can be derived from \Cref{thm:grc}, which, for the case when $L$ has zero rows, gives 
\[\grc(C  x^M )= \rowspan\,[\,M\mid \id_n\,]   \cdot \Trop(\langle Cy \rangle)\,.\]
It now remains  to show that
\[  \rowspan\,[\,M\mid \id_n\,]  \cdot \Trop(\langle Cy \rangle) =\deg(\phi_M) \big(\rowspan(M)\cdot \Trop(\langle Cy \rangle)\big)\,,\]
where we view $\Trop(\langle Cy \rangle)$  in $\RR^{r+n}$ on the left-hand side and in $\RR^r$ on the right-hand side. 
The key step in this is the computation of the relevant sublattice indices in \eqref{eq:intersection_multiplicities},
using the fact that, for $\sigma=\rowspan(M) \cap \ZZ^{r}$ and $\tau=\rowspan_\ZZ(M)$, it holds that
$[\ZZ^{r}_{\sigma} : \ZZ^{r}_{\tau}]  = \deg(\phi_M)$.
 \end{remark}

\begin{remark}[The parameter-separating presentation]\label{rem:vertical}
A convenient property of the parameter-separating presentation $(\overbar{C}\diag(a),\overbar{M})$ is that we may specialize all parameters $a_i$ to 1 in \Cref{thm:monomial_reembedding_bounds,thm:grc,thm:bound_purevertical}, since for any $a\in(\CC^*)^m$, it holds that
\[\mathcal{M}^*\begin{bmatrix}
    \, \overbar{C} \diag(a) & 0 & 0 \\
    0 & L & -b
\end{bmatrix}=\mathcal{M}^*\begin{bmatrix}
    \, \overbar{C} & 0 & 0 \\
    0 & L & -b
\end{bmatrix}\, . \]
\end{remark}

\begin{example}\label{ex:running_grc}
Applying \Cref{alg:bounds} to the system \eqref{eq:1-site_equations}, we obtain that the generic root count is three. If the matrix $L$ in \eqref{eq:minimal_presentation_1-site} is replaced by a matrix with uniform matroid, then  the generic root count is four. By \Cref{cor:generic_degree}, this is the generic degree of the ideal $\langle C x^M\rangle$.
\end{example}

\begin{remark} \label{rem:alpha}
The generic root count of systems of the form $(C x^M - c,Lx-b)$, where $Cx^M$ is a vertically parametrized system and $c$ is a \emph{fixed} constant term, can also be computed using 
tropical intersection numbers  as in \Cref{thm:monomial_reembedding_bounds,thm:bound_purevertical}. 
In particular, to prove \Cref{thm:toric_root_bound} below, we will use that if $C$ and $c$  are real, then
\[
 \resizebox{\linewidth}{!}{$\# \big((\rowspan(M)+h)\cap \Trop^+(\langle C_{a^*}y-c\rangle)\big)\leq    \mrc_{>0}(C x^M - c) 
    \leq \rowspan(M)\cdot \Trop(\langle Cy-c\rangle),$}
\]
  for  any $a^*\in  \RR^m_{>0}$ such that $\Trop(\langle C_{a^*}y-c\rangle)$ takes its generic value, 
and $h\in\QQ^r$ such that 
$\rowspan(M)+h$ and $\Trop(\langle C_{a^*}y-c\rangle)$ intersect transversely.

Moreover,  for any vertical matrix $C$ and vector $c$,
  the generic root count is given by
\[ \grc(C x^M - c)= \deg(\phi_M) \big( \rowspan(M)\cdot \Trop(\langle Cy-c\rangle)\big)      \,,  \]
and it also holds that  $\grc ( \langle Cy-c\rangle + I(\im(\phi_M))  )=\rowspan(M)\cdot \Trop(\langle Cy-c\rangle)$. 

To derive these bounds from \Cref{thm:bound_purevertical}, we consider the matrices $\widehat{M}\in \ZZ^{n\times (r+1)} $  obtained by appending a zero column to $M$, and 
 $\widehat{C} = [ C \mid - a_{r+1} c] \in \CC[a]^{n\times (r+1)}$ for an additional parameter $a_{r+1}$. The vertical system $F=\widehat{C}  x^{\widehat{M}}$ now satisfies that $V(F_{(a,1)})=V(C_a x^M - c)$ for all $a\in \CC^r$, and since 
 scaling $F$ does not change the zero set, 
  we obtain that $\grc(F)=\grc(C x^M - c)$. 

Let $\widehat{I}^\lin=\langle \widehat{C}   z \rangle$ with $z=(z_1,\dots,z_{r+1})$, and let $\pi$ denote the projection onto the first $r$ coordinates. For all $a\in \CC^r$, 
\Cref{thm:TropLinearSpace} and 
\Cref{rmk:linear_ideal} give  the following equalities:
\begin{align*}
\Trop(\langle  C_a y  - c \rangle) &= \pi\big(\Trop(\widehat{I}_{(a,1)}^\lin)\cap \{w\in \RR^{r+1} \mid w_{r+1}=0\}\big)   \\
\Trop^+(\langle C_a y - c\rangle) &= \pi\big(\Trop^+(\widehat{I}_{(a,1)}^\lin)\cap \{w\in \RR^{r+1} \mid  w_{r+1}=0\}\big)\, . \end{align*}
The bounds now follow from \Cref{thm:bound_purevertical} applied to $F$, first using that $\deg(\phi_M)=\deg(\phi_{\widehat{M}})$ and 
\[ 
\rowspan(\widehat{M})\cdot \Trop(\widehat{I}^\lin)  = 
\rowspan(M)\cdot \Trop(\langle Cy- c\rangle) \]
since the last column of $\widehat{M}$ is zero, 
 and then that $\rowspan(\widehat{M}) + (h,0)$ and  $\Trop(\widehat{I}^{\lin}_{(a^*,1)})$ intersect transversely.
\end{remark}

\section{Improved bounds for systems with special structure}
\label{sec:improvements}
\noindent In the previous section, we presented   bounds for 
the root counts $\grc(F)$ and $\mrc_{>0}(F)$ that are applicable for all augmented vertically parametrized systems $F$. In practice, systems arising from reaction networks (see \Cref{sec:reaction_networks}) often have additional structure that can be exploited to improve computations or provide sharper bounds for 
$\mrc_{>0}(F)$. In this section, we explore two such structures: \emph{matroidal cotransversality} and \emph{parametrically toric systems}.

\subsection{Cotransversality of the coefficient matroid}
\label{subsec:transversality}

In this subsection, we introduce a sufficient condition for when the tropical intersection number of an augmented vertically parametrized system in \Cref{thm:grc} agrees with a mixed volume.  This is an important case, as the mixed volume is much easier to compute in practice.  
The condition is naturally expressed using the notion of cotransversal matroids.

\begin{definition}
    A matroid $\mathcal{M}$ on $[n]$ is said to be \term{cotransversal} if
    $\mathcal{M}=\mathcal{M}^*[Q]$ for a
        symbolic matrix $Q\in\QQ(\alpha_{ij}:i\in[d],j\in[n])^{d\times n}$ for which the $(i,j)$-th entry is either $\alpha_{ij}$ or $0$. 
        In this case, we refer to $Q$ as a \term{cotransversal matrix presentation} for $\mathcal{M}$.
\end{definition}

The supports of the rows of a cotransversal presentation of a matroid $\mathcal{M}$ constitute a \emph{transversal presentation} of the dual of $\mathcal{M}$ in the sense of \cite[Section~1.6]{Oxley} (see also \cite{Ardila2007}). 

Note that if the matroid $\mathcal{M}^*[C]$ of a rank $d$ matrix $C$ is cotransversal, one can always find a cotransversal presentation $Q$ of the matroid with $d$ rows. Furthermore, to decide whether a given matrix $Q$ is a cotransversal presentation of $\mathcal{M}^*[C]$, it is enough to check the vanishing of the maximal minors; see \eqref{eq:matroids_agree}. If $C$ is block diagonal, one can find a cotransversal presentation with the same block diagonal structure.

If we replace $C$ and $[L\mid -b]$ by cotransversal presentations of their matroids, we obtain a system with fixed support and freely varying coefficients, and could therefore expect the generic root count to be a mixed volume (denoted MV). The following theorem formalizes this intuition.

\begin{theorem}[Bound for cotransversal systems]
\label{thm:cotransversal}
    Let $F=(C  x^M  ,Lx-b)\in \CC[a,b][x^\pm]^n$  be a square augmented vertically parametrized system  for a vertical matrix $C\in \CC[a]^{s\times r}$, $L\in \mathbb{C}^{d\times n}$ of full row rank,  $M\in \mathbb{Z}^{n\times r}$, and $n=s+d$. 
    Assume that the generic matroids of $C$ and of $[L \mid -b]$ are cotransversal with respective matrix presentations
    \[ P \in \QQ(\beta_{ij} \mid i\in [s],j\in[r])^{s\times r}, \qquad 
    [Q\ |\  q] \in \QQ(\gamma_{ij} \mid i\in [d],j\in[n+1])^{d\times (n+1)} \,   \]
    where $q$ is a column vector of size $d$. 
    Then
    \[\grc(F) = \MV(P x^M , Q x + q)\,.\]
\end{theorem}
\begin{proof}
Recall the ideals $I^{\bin}$ and $I^{\lin}$ from \eqref{eq:monomial_linbin}. 
We have the series of equalities
\begin{equation*}
    \begin{array}{rclcl}
   \grc(F) 
      &\overset{(1)}{=}&  \Trop(I^{\bin}) \cdot \Trop(I^{\lin})  
      &\overset{(2)}{=}&   \Trop(I^{\bin})  \cdot \Trop( \langle 
      P  y , Q  x + q   \rangle )
      \\ 
          &\overset{(3)}{=}& \grc( P x^M , Q  x + q) 
      &\overset{(4)}{=}& \MV(P x^M , Q  x + q) \, ,
    \end{array}
  \end{equation*}
obtained by the following arguments: 
  \begin{enumerate}
  \item Follows from \Cref{thm:grc}. 
  \item Follows from hypothesis, as the generic matroids of $C$ and $[L\mid-b]$ agree with the matroids of $P_\beta$ and $[Q_\gamma\mid q_\gamma]$ for generic $(\beta,\gamma) \in \CC^{s\times r}\times\CC^{d\times (n+1)}$.
   
    \item Introduce new parameters $c=(c_1,\ldots,c_r)$. For each fixed $(\beta,\gamma)$, the system $(y-x^M, P_\beta \diag(c) y, Q_\gamma  x + q_\gamma)$ is the toric reembedding of the augmented vertically parametrized system  $\big( P_\beta(c\star x^M), Q_\gamma  x + q_\gamma  \big)$. 
    Hence, the equality follows by first applying \Cref{thm:grc} and then realizing that as $\beta$ takes generic values, 
    the value of $c\in (\CC^*)^r$ does not matter for the generic root count. 
\item Follows from Bernstein's theorem, as all nonzero coefficients are freely varying. \qedhere
\end{enumerate}  
\end{proof}

\begin{remark}
Unlike in   \Cref{thm:monomial_reembedding_bounds}, the ideas used in the proof of \Cref{thm:cotransversal} do not immediately allow us to derive lower bounds on $\mrc_{>0}(F)$. One of the main obstacles is that the proof relies on genericity assumptions that do not carry over to the real setting. 
One possible direction to investigate in future work is the connection to \emph{positively decorated} mixed cells, which by \cite[Theorem 3.2]{BihanDickensteinGiaroli2020} give a lower bound on the maximal number of positive zeros of a sparse system with freely varying coefficients with fixed signs. 
\end{remark}

\begin{example}\label{ex:cotransversal}
We revisit  system \eqref{eq:1-site_equations} with the parameter-separating presentation from \eqref{eq:matrices_vertical}. By \Cref{rem:vertical},  the computation of $\Trop(I^{\lin})$ does not depend on the parameter $a$. 
The generic matroids of $C$ and of $[L \mid -b]$ are cotransversal with respective cotransversal presentations
\[
 P = {\small\begin{bmatrix}
\beta_{1,1} & \beta_{1,2} & \beta_{1,3} & 0 & 0 & \beta_{1,6} \\
0& 0 &\beta_{2,3} & \beta_{2,4} & \beta_{2,5} & \beta_{2,6} \\
0 & 0 &\beta_{3,3} & 0 & 0 &\beta_{3,6} 
\end{bmatrix}}, \qquad 
[Q\ | \ q] = 
{\small\begin{bmatrix}
\gamma_{1,1} & \gamma_{1,2} & \gamma_{1,3} & \gamma_{1,4} & 0 & 0 & \gamma_{1,7} \\
\gamma_{2,1} & 0 & 0 & 0 & \gamma_{2,5} & 0 & \gamma_{2,7} \\
0 & \gamma_{3,2} & 0 & 0&0 &  \gamma_{3,6} & \gamma_{3,7}
\end{bmatrix}}\, . 
\]
By \Cref{thm:cotransversal}, the generic root count is given by the mixed volume of the following system (we specialize all the parameters to $1$ for readability)
\begin{multline*}
   \qquad  (
x_1x_3 + x_5  +  x_6,\ x_5 +  x_2x_4 +    x_6 ,\  x_5 +  x_6,\\
x_1 + x_2 +  x_3 + x_4  + 1, \ 
x_1  + x_5 + 1  , \ 
 x_2 +  x_6 + 1 
)\,, \qquad 
\end{multline*}
which equals three (in agreement with the result given in \Cref{ex:running_grc}).
\end{example}
 
\subsection{Parametrically toric vertical systems}
\label{subsec:parametric_toricity}

An important feature of many augmented vertical systems $F=(Cx^M, Lx-b)$ appearing in the biochemical literature is that the set of positive zeros $V_{>0}(C_ax^M)$ of the vertical part admits a monomial parametrization, in the sense that it is the positive part of a scaled toric variety \cite{FeliuHenriksson2024}. Here, we discuss how this property can be exploited to compute tropical bounds for the number of positive zeros that are faster to compute than those in \Cref{thm:monomial_reembedding_bounds}, and in some cases also sharper.

Consider the (positive) \emph{feasibility locus} of $Cx^M$, defined as
\[\mathcal{Z}_{>0}:=\{a\in \RR^m_{>0}\mid \text{$C_a  x^M=0$ for some $x\in\RR^n_{>0}$}\}\,.\]
We say that $F$ is
\term{parametrically toric} with respect to a given exponent matrix $A\in\ZZ^{d\times n}$ if there exists a function  $\psi\colon\mathcal{Z}_{>0}\to \RR^n_{>0}$
such that
\begin{equation}\label{eq:psi}
    V_{>0}(C_a x^M) =\psi(a)\star\mathcal{T}_A\qquad\text{for all $a\in\mathcal{Z}_{>0}$}\, ,
    \end{equation}
where we recall that $\mathcal{T}_A$ is the image of the monomial map $\phi_A(t)=t^A$ restricted to $\RR^d_{>0}$. 

\begin{example}
\label{ex:running_example_is_toric}
For our running example \eqref{eq:1-site_equations}, a straightforward computation shows that $\mathcal{Z}_{>0}=\RR^6_{>0}$ and that the zeros of the vertical part admit the parametrization
$$V_{>0}(C_ax^M)=\left\{\big(t_1,t_2,t_3,\tfrac{a_{1} a_{3}(a_{5}+a_{6})}{(a_{2}+a_{3})a_{4}a_{6}} t_1t_2^{-1}t_3, \tfrac{a_{1}}{a_{2}+a_{3}} t_1t_3, 
\tfrac{a_{3} a_{1}}{(a_{2}+a_{3})a_{6}} t_1t_3
\big)\mid t\in\RR^3_{>0}\right\}.$$
Hence, the system is parametrically toric with respect to the matrix
\[A=\begin{bmatrix}1 & 0 & 0 & 1 & 1 & 1\\
0 & 1 & 0 & -1 & 0 & 0\\
0 & 0 & 1 & 1 & 1 & 1\end{bmatrix}\qquad \text{using} \quad \psi(a)=\left(1,1,1,\tfrac{a_{1} a_{3}(a_{5}+a_{6})}{(a_{2}+a_{3}) a_{4}a_{6} }, \tfrac{a_{1}}{a_{2}+a_{3}}, 
\tfrac{a_{3} a_{1}}{(a_{2}+a_{3})a_{6}} \right)\,.\qedhere\]
\end{example}

\begin{remark}
\label{rmk:nonemptiness_of_Z}
The feasibility locus $\mathcal{Z}_{>0}$ is nonempty if and only if $\ker \overbar{C} \cap \RR^m_{>0}\neq \emptyset$ for the matrix $\overbar{C}$ in the parameter-separating presentation \eqref{eq:parameter_separating}. In this case, it holds that any $c\in\RR^n_{>0}$ is a zero of $C_{a^*}x^M$ by letting $a^*\coloneqq v \star c^{-\overbar{M}}$ with $v \in \ker(\overbar{C}) \cap \RR^m_{>0}$.
\end{remark}

Special cases of this type of toric structure have been studied in reaction network theory since the 1970s \cite{HornJackson1972}, and have more recently appeared under names such as \emph{toric dynamical systems} \cite{CraciunDickensteinShiuSturmfels}, \emph{disguised toric dynamical systems} \cite{BrustengaCraciunSorea}, and \emph{toric steady states} \cite{PerezMillanDickensteinShiuConradi}. 
In \cite{FeliuHenriksson2024}, parametrically toric vertical systems are studied in full generality through the lens of affine dependencies between the columns of $\overbar{M}$ and the matroid ~$\mathcal{M}^*[\overbar{C}]$. 
In particular, it is shown that $A$ can easily be computed via linear algebra in terms of $\overbar{C}$ and $\overbar{M}$.

The following is the main result of this subsection. We let $\conv(A,0)$ be the convex hull of the columns of $A$ and the origin in $\RR^d$, and $\vol(\cdot)$ denote the normalized volume of a polytope. As in \Cref{thm:bound_purevertical}, we view $\rowspan(A)$ as a single polyhedron of multiplicity~1.

\begin{theorem}[Toric root bounds]
\label{thm:toric_root_bound}\label{thm:toric_and_transversal}
Let $F=(C x^M, Lx-b)$ be a
square augmented vertically parametrized system for 
a vertical matrix $C\in \RR[a]^{s\times r}$, $L\in \mathbb{R}^{d\times n}$ of full row rank, and $M\in \mathbb{Z}^{n\times r}$. 
Suppose that the vertical part is parametrically toric with respect to a matrix $A\in\ZZ^{d\times n}$ of rank $d$. 
Let $J^{\lin} \coloneqq \langle\,L x- b\,\rangle \subseteq \CC(b)[x^\pm]$ and assume $\mathcal{Z}_{>0}\neq\emptyset$. 

\begin{enumerate}[label=(\roman*)]
\item  For any $b \in L(\RR^n_{>0})$ and $h \in \QQ^n$ such that $\rowspan(A) + h$ and  $\Trop(\langle Lx-b\rangle)$ intersect transversely, it holds that 
\[ \#\Big( \big(\rowspan(A) + h\big) \cap \Trop^+(\langle Lx-b\rangle) \Big) \leq \mrc_{>0}(F)\leq \rowspan(A)\cdot \Trop(J^{\lin})\,. \]
 \item If the generic matroid of $[\,L\mid -b\,]$
      is cotransversal with cotransversal presentation $[\,Q\ | \ q\,] \in \QQ(\gamma_{ij})^{d\times (n+1)}$, then  
\[\rowspan(A)\cdot \Trop(J^{\lin})= \MV(Q z^A +q)/\deg(\phi_A)\,.  \]
The right-hand side of this equality simplifies to $\vol(\conv(A,0))/\deg(\phi_A)$ when the generic matroid of $[\,L\mid -b\,]$ is uniform of rank $d$.   
\end{enumerate}
\end{theorem}
\begin{proof}
By hypothesis, $x\in V_{>0}(C_a x^M)$ if and only if $x=\psi(a) \star z^A$ for some $z\in \RR^d_{>0}$. 
Therefore, as  $\phi_A$ is injective on $\RR^d_{>0}$, there is a one-to-one correspondence between $V_{>0}(F_{a,b})$ and $V_{>0}(L( \psi(a) \star z^A) - b)$. 

Fix $b$ and consider the  parametrized system
$G=L( c \star z^A) - b$ with $c=(c_1,\dots,c_{n})$. As in \Cref{rem:vertical}, the tropicalization of $\langle L( c \star x) - b \rangle $   is simply $\Trop(\langle L  x - b\rangle)$. 
Using \Cref{rem:alpha}, 
\begin{align*}
    \mrc_{>0}(G) 
    \leq  \rowspan(A)\cdot \Trop(\langle Lx-b\rangle)\, .
\end{align*}
The right-hand side is maximized when $b$ is such that $\Trop(\langle Lx-b\rangle)$ takes its generic value. Hence the second inequality in part (i) holds.  
 
To show the first inequality in part (i), we use that 
\Cref{rem:alpha}  gives that
\begin{align}\label{eq:ineq}
\mrc_{>0}(G)   \geq \# \big((\rowspan(A)+h)\cap \Trop^+(\langle\,Lx-b\,\rangle)\big)\,.
\end{align}
Pick $c\in \RR^n_{>0}$ such that $\# V_{>0}(G_{c})$ is at least equal to the right-hand side of \eqref{eq:ineq}. Using now \Cref{rmk:nonemptiness_of_Z}, there exists  $a^*\in \RR^m_{>0}$ such that 
\[ c \in V_{>0}(C_{a^*}\, x^M) =\psi(a^*)\star\mathcal{T}_A \, \quad \Rightarrow \quad c\star\mathcal{T}_A =  \psi(a^*)\star\mathcal{T}_A\,. \]
Hence 
$V_{>0}(G_{c})    = V_{>0}(L(\psi(a^*) \star z^A)-b)  = V_{>0} (F_{a^*,b})$ and the claim follows using \eqref{eq:ineq}.

Part (ii)  is obtained by the following equalities, similar to the proof of \Cref{thm:cotransversal}:
\begin{align*}
& \rowspan(A)\cdot \Trop(J^{\lin})\,    && 
\\ 
&\quad = \rowspan(A)\cdot  \Trop(\langle\, Q  x + q \,\rangle)  && \text{(by cotransversality)}  \\
&\quad = \grc(Q_\gamma (\lambda \star z^A) + q_\gamma   )  / \deg(\phi_A) && \text{(with free parameters $\lambda,\gamma$, by \Cref{rem:alpha})} \\ 
&\quad = \grc(Q  z^A + q   )  / \deg(\phi_A)   && \text{(trivial)}
\\ &\quad = \MV(Q z^A + q)/\deg(\phi_A) && \text{(by Bernstein's theorem)}\, . 
\end{align*}
If  $\mathcal{M}^*[\,L\mid -b\,]$ is uniform, then it admits a cotransversal presentation $[\,Q\ | \ q\,]$ having all entries nonzero. 
Hence,  $\MV(Qz^A+q)$ is simply $\vol(\conv(A,0))$ by Kushnirenko's theorem. 
\end{proof}
 
Note that it is not necessary to have an explicit description of the map $\psi$ to apply the bounds of \Cref{thm:toric_root_bound}.
We refer to the upper bound in \Cref{thm:toric_root_bound}(i) as the \term{upper toric root bound} of $F$. This bound offers two advantages compared to the bound in \Cref{thm:monomial_reembedding_bounds}:
\begin{enumerate}
    \item The stable intersection takes place in $\RR^n$, rather than $\RR^{r+n}$, which can be expected to speed up the calculations.
    \item\label{it:sharper_bound} If $V_{\CC^*}(C_a  x^M)$ is generically reducible, we get a sharper upper bound of $\mrc_{>0}(F)$, as the upper toric root bound only takes into account the component intersecting $\RR^n_{>0}$.
\end{enumerate}

\begin{example}
Consider the augmented vertically parametrized system 
\[  F=((a_1-a_2)x_1^3x_2^2 + a_3x_2^4- 2a_4x_1^6,\  2x_1+3x_2-b)\,  \]
whose vertical part is parametrically toric with respect to the exponent matrix $A=[\,2\ \ 3\,]$; see \cite[Eq~(2.4)]{FeliuHenriksson2024}. 
A simple calculation reveals that $\grc(F)=6$
and that the upper toric root bound from \Cref{thm:toric_root_bound}(i) is 
\[\mrc_{>0}(F)\leq \rowspan([\,2\ \ 3\,])\cdot \Trop(\langle\, 2x_1+3x_2-b \,\rangle)=3\, ,\]
where $b$ is chosen to be generic. This improvement over the bound $\mrc_{>0}(F)\leq  \grc(F)$ is explained by the fact that the vertical part $V(C  x^M)$ generically consists of two irreducible components, each of which generically intersects $V(Lx-b)$ three times.

    For $h=(1,0)$ and $b = 1$, the intersection
    \[ \big(\rowspan([\,2\:\:3\,]) + (1,0)\,\big) \cap \Trop^+(\langle 2x_1+3x_2-1\rangle) \]
    is transverse and contains one point. Thus $\mrc_{>0}(F) \geq 1$ by \Cref{thm:toric_root_bound}(i). By the results of \cite[Section~3]{MullerFeliuRegensburgerConradiShiuDickenstein2015} (see also \cite[Proposition~7.1]{FeliuHenriksson2024}), the true value of $\mrc_{>0}(F)$ is $1$.
\end{example}

\begin{example}
\label{ex:toric_bounds_for_running_example}
Our running example \eqref{eq:1-site_equations}  is parametrically toric with respect to the exponent matrix $A$ found in \Cref{ex:running_example_is_toric}, which has $\deg(\phi_A)=1$. The generic matroid of $[\,L\mid -b\,]$
is cotransversal with presentation $[\,Q\mid q\,]$ given in \Cref{ex:cotransversal}. By \Cref{thm:toric_and_transversal}(ii), an upper bound of the maximal positive root count is  the mixed volume of $Q z^A + q$, which is three. 
\end{example}

\section{Applications in reaction network theory}
\label{sec:reaction_networks}
A reaction network on a set of \emph{species} $\{X_1,\dots,X_n\}$ is a collection of labeled \emph{reactions}
\[\alpha_{1j}X_1+\dots+\alpha_{nj}X_n \ce{->[$a_j$]} \beta_{1j} X_1+\dots+\beta_{nj} X_n, \qquad j=1,\dots,m\]
for $\alpha_{ij},\beta_{ij}\in\ZZ_{\geq 0}$ and $a_j\in \RR_{>0}$. 
Under the assumption of \emph{power-law kinetics}, the concentrations of the species at time $t$, denoted $x(t)=(x_1(t),\dots,x_n(t))$, can be modeled by a system of autonomous ordinary differential equations (ODEs) of the form
\[ \frac{d x}{dt} = N (a \star x^M), \qquad  x \in \RR^n_{>0} \, ,\]
where:
\begin{itemize}
    \item $N = (\beta_{ij}-\alpha_{ij}) \in \ZZ^{n\times m}$ is the \emph{stoichiometric matrix}, which encodes the net production of each species in each reaction,
    \item $a=(a_1,\dots,a_m) \in \RR^m_{>0}$ is the vector of \emph{rate constants},
    \item $M\in \ZZ^{n\times m}$ is the \emph{kinetic matrix} of the model. A common modeling assumption is to take $M=(\alpha_{ij})$, in which case the system is said to have \emph{mass-action kinetics}. 
\end{itemize}

The \emph{steady states} of the ODE system  are the zeros in $\RR^n_{>0}$ of the (Laurent) polynomial system $N (a \star x^M)$. If $s=\rank(N)<n$, then 
we pick a matrix $C\in \RR^{s\times m}$ whose rows are linearly independent vectors in $\rowspan(N)$ and form the system
\[C (a \star x^M)  \in \RR[x^\pm]^s  \]
whose positive zeros are the steady states. 
None of the results below depend on the choice of $C$. In this case the ODE system is not full dimensional and trajectories are confined to proper affine linear subspaces, which are parallel translates of $\im(N)$. The intersections of these sets with the nonnegative orthant are called \emph{stoichiometric classes}. 
By letting $d=n-s$ and $L\in \RR^{d \times n}$ of full rank such that $L N =0$, the stoichiometric classes are given by equations
\[ L x - b =0 \, , \]
with $b=(b_1,\dots,b_d)\in \RR^{d}$ called the vector of \emph{total amounts}.

By all the above, the  steady states of the network within a stoichiometric class are the positive zeros of the 
 \term{(parametric) steady state system}: 
 \begin{align}
    \label{eq:SteadyStateSystem}
  F=\left(C (a \star x^M),\: L x - b\right)\, \in \RR[a,b][x^\pm]^n\,,
    \end{align}
    which is an augmented vertically parametrized system, given naturally with the parameter-separating presentation.

For a reaction network with steady state system $F$, we will refer to $\grc(F)$ as the \term{steady state degree} (in the torus) of the network. We note that $\mrc_{>0}(F)>1$ if and only if the network has the capacity for (nondegenerate) multistationarity, that is,  \eqref{eq:SteadyStateSystem} has at least two nondegenerate positive zeros for some $(a,b)\in\RR^m_{>0}\times\RR^d$. Steady state degrees have previously been computed using Gröbner bases \cite{Gross2016wnt} and bounded above by mixed volume computations \cite{GrossHill2021}.
 
The running example presented in \Cref{ex:running_matrices} arises from the following phosphorylation and dephosphorylation network, which is ubiquitous in the literature:
\begin{equation}
\label{eq:1-site_network}
S_0 + K \ce{<=>[$a_1$][$a_2$]} S_0K \ce{->[$a_3$]} S_1 + K  \qquad 
    \ce{$S_1$ + $P$ <=>[$a_4$][$a_5$] $S_1P$ ->[$a_6$] $S_0$ + $P$}\,.
\end{equation}
The chemical meaning of the species is as follows: $S_0$ is a biochemical substrate, $S_1$ is a phosphorylated form of the substrate, $K$ is an enzyme (kinase) that catalyzes the phosphorylation of $S_0$, $P$ is another enzyme (phosphatase) that catalyzes the dephosphorylation of $S_1$, and $S_0K$ and $S_1P$ are intermediate protein complexes. If $x=(x_1,\ldots,x_6)$ denotes the concentration of the species ordered as $(K,P,S_0,S_1,S_0K,S_1P)$, the steady state system is the system $F$ in \eqref{eq:1-site_equations}: the first three entries of $F$ correspond to $S_1$, $S_0K$ and $S_1P$, respectively, and the last three entries define the stoichiometric classes. 

We saw in \Cref{ex:cotransversal} that the generic root count could be found by a simple mixed volume computation. The same holds for the general class of \emph{multisite phosphorylation networks}. Specifically, for any positive integer $k$, the \emph{$k$-site phosphorylation network} is given by
\begin{align}\label{eq:ksite}
\begin{split}
S_{i} + K \:\xrightleftharpoons[a_{6i+2}]{a_{6i+1}}\: & S_{i}K \xrightarrow{a_{6i+3}} S_{i+1} + K  \\
S_{i+1} + P \:\xrightleftharpoons[a_{6i+5}]{a_{6i+4}}\: & S_{i+1}P \xrightarrow{a_{6i+6}} S_{i} + P 
\end{split}
\qquad \quad i =0,\dots,k-1\,,
\end{align}
where $k$ is the number of phosphorylation sites of the network. It was shown in \cite{WangSontag} that the steady state degree is $2k+1$ (as $F$ generically has  no zeros in the coordinate hyperplanes), that  the number of  steady states is bounded from above by $2k-1$, and that at least $R\coloneqq 2\lfloor \tfrac{k}{2} \rfloor + 1$ of them are obtainable for some choice of parameters. 
The number $R$ is achieved in the open parameter regions discussed in \cite{GiaroliRischterPerezDickenstein}, and   in \cite{rendall:unlimited} it is shown that $R$ nondegenerate steady states are obtained, with $\lfloor \tfrac{k}{2} \rfloor + 1$ of them being asymptotically stable and the rest unstable. 
The bound $2k-1$ is achieved for $k\in\{2,3,4\}$; see \cite{FlockerziHolsteinConradi}. A proof that  the  bound $2k - 1$  is achievable for all $k$ has recently been announced by Dickenstein and Fattori \cite{Fattori}. 

The steady state systems of the $k$-site phosphorylation networks are parametrically toric \cite{PerezMillanDickensteinShiuConradi}, so \Cref{thm:toric_root_bound}(i) applies. Additionally, we show now that they give rise to cotransversal matroids, so   \Cref{thm:cotransversal} and \Cref{thm:toric_and_transversal}(ii) are also applicable.

\begin{theorem}\label{thm:cotransversal_ksite}
Consider the parametric steady state system $F=(C (a \star x^M),\: L x - b)$ in \eqref{eq:SteadyStateSystem} of the 
$k$-site phosphorylation network for $k\geq 1$. The matroids of $C$ and $[ L \mid -b]$ for generic $b$ are cotransversal.
\end{theorem}

\begin{proof}
    By ordering the species as $(K,P,S_0,\dots,S_k,S_0K,\dots,S_{k-1}K,S_1P,\dots,S_kP)$,
a matrix $L$ describing the stoichiometric classes is 
\[ [L \mid -b]  ={\small \left[ \begin{array}{ccccc|ccc|ccc|c}
0 & 0 & 1 & \dots & 1 & 1 & \dots & 1 & 1 & \dots & 1 & -b_1\\
1 & 0 & 0 & \dots & 0 & 1 & \dots & 1 & 0 & \dots & 0 & -b_2\\
0 & 1 & 0 & \dots & 0 & 0 & \dots & 0 & 1 & \dots & 1 & -b_3
\end{array}\right]} \]
(see e.g. \cite[Equation (4.1)]{PerezMillanDickensteinShiuConradi}). Each row corresponds to the total amounts of $S$, $K$ and $P$,  respectively. The first block is for the species $K,P,S_0,\dots,S_k$, the second for $S_0K,\dots,S_{k-1}K$ and the last for $S_1P,\dots,S_kP$. By subtracting the second and third rows from the first, we obtain the following matrix with the same matroid as $[L \mid -b]$:
\[ L'\coloneqq {\small \left[ \begin{array}{ccccc|ccc|ccc|c}
-1 & -1 & 1 & \dots & 1 & 0 & \dots & 0 & 0 & \dots & 0  & -b_1+b_2+b_3 \\
1 & 0 & 0 & \dots & 0 & 1 & \dots & 1 & 0 & \dots & 0 & -b_2 \\
0 & 1 & 0 & \dots & 0 & 0 & \dots & 0 & 1 & \dots & 1 & -b_3
\end{array}\right]}\, . \]
The  generic matroid of $L'$ coincides with that of the following matrix $[Q \mid q]\in \CC[\gamma_{ij}]^{3\times (3k+4)}$:
\[  {\small \left[ \begin{array}{ccccc|ccc|ccc|c}
\gamma_{1,1} & \gamma_{1,2}  & \gamma_{1,3}  & \dots & \gamma_{1,k+3}  & 0 & \dots & 0 & 0 & \dots & 0 & \gamma_{1,3k+4}\\
\gamma_{2,1} & 0 & 0 & \dots & 0 & \gamma_{2,k+4} & \dots & \gamma_{2,2k+3} & 0 & \dots & 0 & \gamma_{2,3k+4}\\
0 & \gamma_{3,2} & 0 & \dots & 0 & 0 & \dots & 0 & \gamma_{3,2k+4} & \dots & \gamma_{3,3k+3} & \gamma_{3,3k+4}
\end{array}\right]}\, .  \]
Indeed, using  the structure of $[Q \mid q]$, it suffices to show that the   matrices
\[ {\small \left[ \begin{array}{ccccc|c}
\gamma_{1,1} & \gamma_{1,2}  & \gamma_{1,3}  &  0 & 0 & \gamma_{1,3k+4}\\
\gamma_{2,1} & 0 & 0 & \gamma_{2,k+4}   & 0  & \gamma_{2,3k+4}\\
0 & \gamma_{3,2} & 0   & 0  &  \gamma_{3,2k+4}   & \gamma_{3,3k+4}
\end{array}\right]\, \quad\text{and}\quad \left[ \begin{array}{ccccc|c}
-1 & -1  & 1  &  0 & 0 & \gamma_{1,3k+4}\\
1 & 0 & 0 & 1  & 0  & \gamma_{2,3k+4}\\
0 & 1 & 0   & 0  & 1  & \gamma_{3,3k+4}
\end{array}\right]  } \]
have the same generic matroid. This can be checked by a direct computation using, for example, \Cref{lem:genericMatroid}. Therefore, the generic matroid of $[L\mid -b]$ is cotransversal with cotransversal presentation  $[Q \mid q]$.

We now construct $C$. By the form of $L$, and since the stoichiometric matrix $N$ has rank $3k$ and $3k+3$ rows, each corresponding to a species, $C$ can be chosen to consist of the last $3k$ rows of $N$. Each column of $C$ corresponds to a reaction, and using the  numbering of the reactions in \eqref{eq:ksite}, we   write $C=[C_0\mid \dots \mid C_{k-1}]$ where block $C_i$ consists of the reactions $6i+1,\dots,6i+6$. 

Consider the following matrices:
\[   {\footnotesize W_1\coloneqq\begin{bmatrix}
    -1 & 1 & 0 & 0 & 0 & 1 \\
    0 & 0 & 1 & -1 & 1 & 0 \\
    1& -1 & -1 & 0 & 0 & 0 \\
     0 & 0 & 0 & 1& -1 & -1 
\end{bmatrix}\,, \qquad W_2\coloneqq\begin{bmatrix}
     0 & 0 & -1 & 0 & 0 & 1 \\
    0 & 0 & 1 & 0 & 0 & -1\\
    1& -1 & -1 & 0 & 0 & 0 \\
     0 & 0 & 0 & 1& -1 & -1 
\end{bmatrix}\,. \quad 
}\]
For $i>0$, the block $C_i$ has $4$ nonzero rows, corresponding to the species $S_i,S_{i+1},S_iK,S_{i+1}P$. These rows form the submatrix $W_1$. 
In particular, the rows of $S_iK,S_{i+1}P$ are zero outside $C_i$, 
 the row of $S_i$ has nonzero entries only in  $C_i,C_{i-1}$ if $i<k$, and  in   $C_{k-1}$ if $i=k$. 
We now modify $C$   to obtain a block matrix with $k$ blocks and the same row span. First, we add, for all $i>0$, the row  of $S_iK$ to that of $S_i$, and the row of $S_{i+1}P$ to that of   $S_{i+1}$, so these four rows in $C_i$ become $W_2$. 
Next, we add to the row of $S_i$ ($i>0$), the sum of the rows of $S_{i+1},\dots,S_k$, and obtain a new matrix $\widetilde{C}=[C_0 \mid \widetilde{C}_1\mid \dots \mid \widetilde{C}_{k-1}]$ with the same row span as $C$. The block 
$C_0$ is exactly the coefficient matrix of \eqref{eq:1-site_equations}, with additional $3k-3$ zero rows below. The block $\widetilde{C}_i$ for $i>0$ is only nonzero in the rows of $S_{i+1}, S_iK, S_{i+1}P$, which agree with the bottom 3 rows of $W_2$. 
So $\widetilde{C}$ is a block matrix. 
As shown in \Cref{ex:cotransversal}, the matroid of $C_0$ is cotransversal. Let $P_0$ denote the cotransversal presentation.
By a direct computation, the matroid of $\widetilde{C}_i$ for $i>0$ is cotransversal with cotransversal presentation
\[ P_i \coloneqq {\small\begin{bmatrix}
        0 & 0 & w_{i,1} & 0 & 0 & w_{i,2}\\
    w_{i,3} & w_{i,4} & w_{i,5} & 0 & 0 & 0 \\
     0 & 0 & 0 &  w_{i,6} & w_{i,7} & w_{i,8} 
\end{bmatrix}}.\]
We conclude that $C$ admits the cotransversal presentation $P:=[P_0\mid P_1 \mid \dots \mid P_{k-1}]$.
\end{proof}

It follows from \Cref{thm:cotransversal_ksite,thm:cotransversal} that the steady state degree of the $k$-site phosphorylation network can be obtained via a mixed volume computation.
The proof of \Cref{thm:cotransversal_ksite} also shows how to construct the cotransversal presentations appearing in \Cref{thm:cotransversal}.
By \cite{WangSontag}, this mixed volume must equal $2k+1$. 

\section{A Julia implementation}
\label{sec:implementation}
\noindent We close the paper by presenting our proof-of-concept implementation of the complex and positive root bounds derived in the previous sections. Our Julia package \texttt{VerticalRootCounts.jl} is accessible at
\begin{center}
\url{https://github.com/oskarhenriksson/VerticalRootCounts.jl}\,.
\end{center}
The package is built on top of the computer algebra system \texttt{OSCAR} \cite{OSCAR-book}, with an interface to the chemical reaction network theory package \texttt{Catalyst.jl} \cite{Catalyst}.

By default, the package uses three strategies for generic root count computations. It first checks if the rank condition \eqref{eq:rank_condition} holds, in which case it returns zero. It then checks if the linear part admits a cotransversal presentation, and if so, computes the generic root count as a mixed volume via \Cref{thm:cotransversal}. Otherwise, it computes the stable intersection of the tropical linear space and the tropicalized binomial variety as in \Cref{alg:bounds}. 
All auxiliary data (e.g., the certifiably generic choice of parameters, cotransversal presentations, and perturbations in stable intersections) are returned along with the root counts. 
For positive lower root bounds, the package makes a user-specified number of specializations of the parameters $(a,b)$ and perturbations of the tropical varieties, and returns the optimal bound together with these choices. The package also supports parametrically toric systems (with a known exponent matrix $A$) and computes the upper and lower root bounds from \Cref{thm:toric_root_bound}.

In the non-cotransversal case, the two main bottlenecks are the computation of the tropical linear space and the stable intersection. For the former, we rely on the algorithm for computing Bergman fans of realizable matroids from \cite{Rincon13}, via the implementation \cite{HampeJoswig17} in \textsc{polymake} \cite{polymake}. The output of this algorithm is a simplicial polyhedral fan that is a refinement of the polyhedral complex described in \Cref{thm:TropLinearSpace} and which, in practice, is substantially faster to compute \cite{HampeJoswigSchroeter:2019}. 

For the cotransversal case, the main bottlenecks are instead the verification of cotransversality and the computation of the mixed volume. 
Computing (maximal) transversal presentations  of matroids goes back to Brualdi and Dinolt \cite{BrualdiDinolt:1972}. We check cotransversality by computing the $\beta$-invariant of the cyclic flats of the matroid as described in \cite{BKM11}, and again we rely on the implementation in \textsc{polymake}.
Note that this computation is costly and currently does not exploit the fact that our matroids are associated with a matrix, which leaves room for improvement.
For computing mixed volumes, we use the Julia package \texttt{MixedSubdivisions.jl}
\cite{mixedsubdivisions}, which uses the tropical homotopy continuation algorithm of \cite{Jensen16}.

We end the manuscript by reporting on the performance of our package for the multisite phosphorylation family \eqref{eq:ksite}. As shown in \Cref{tab:counts_multisite}, our package correctly computes the steady state degree $2k+1$ for $k\leq 9$. For the positive root bounds, it is expected by \cite[Proposition~6.8]{RoseTelek24} that we should recover the lower bound $2\lfloor\frac{k}{2}\rfloor+1$. That we do not reach this bound in our experiments, and in particular do not attain the   upper bound $2k-1$, is explained by \Cref{rem:unbounded_regions}. 
We display the running times\footnote{The computations were performed on a MacBook Pro with an M4 chip and 18 GB RAM. Each entry is an average of three runs.} for the complex generic root count, using different computational strategies in  \Cref{tab:times_multisite}, which highlights the computational gains of taking into account additional matroidal and toric structure that an augmented system might possess, while also showing the tradeoff with the computational cost of detecting this structure.

\begin{table}[t]
\caption{Computed root counts for $k$-site phosphorylation networks.
}
\label{tab:counts_multisite}
\centering
\footnotesize
\begin{tabular}{lccccccccc}
\toprule
   Number of phosphorylation sites $k$ & 1 & 2 & 3 & 4 & 5 & 6 & 7 & 8 & 9  \\
      \midrule
Number of variables ($n=3k+3$) & 6 & 9 & 12 & 15 & 18 & 21 & 24 & 27 & 30\\
Number of parameters ($m+d=6k+3$)\quad & 9 & 15 & 21 & 27 & 33 & 39 & 45 & 51 & 57  \\\midrule
Steady state degree & 3 & 5 & 7 & 9 & 11 & 13 & 15 & 17 & 19\\
Best computed lower bound of $\mrc_{>0}$   & 1 & 3 & 3 & 5 & 5 &  5 & 5 & 3 & 3\\
  \bottomrule
\end{tabular}
\end{table}

\begin{table}[tb]
\caption{Running time in seconds for root bounds for $k$-site phosphorylation networks. The second and fourth row include the time for finding a cotransversal presentation. Empty entries indicate that the computation did not terminate within 3 hours.}
\label{tab:times_multisite}
\centering
\footnotesize
\begin{tabular}{lrrrrrrrrr}
\toprule
   Number of phosphorylation sites $k$ & \multicolumn{1}{c}{1} & \multicolumn{1}{c}{2} & \multicolumn{1}{c}{3} & \multicolumn{1}{c}{4} & \multicolumn{1}{c}{5} & \multicolumn{1}{c}{6} & \multicolumn{1}{c}{7} & \multicolumn{1}{c}{8} & \multicolumn{1}{c}{9}  \\
      \midrule
Stable intersection (\Cref{alg:bounds}) & 0.04 & 1.41 & 153.36\\
Using cotransversality (\Cref{thm:cotransversal}) & 0.02 & 0.05 & 0.06 & 0.48 & 10.25 & 409.79 & \\
Using  toric structure (\Cref{thm:toric_root_bound}(i))
& 0.02 & 0.04 & 0.34 & 0.96 & 3.92 & 13.48 & 40.47 & 107.59 & 262.71\\
Using both structures (\Cref{thm:toric_and_transversal}(ii)) & 0.01 & 0.01 & 0.01 & 0.04 & 0.30 & 3.18 & 34.88 & 372.87 & 3946.74\\
  \bottomrule
\end{tabular}
\end{table}

\printbibliography

\medskip

\goodbreak
{\small
\noindent {\bf Authors' addresses:}\\
\noindent 
Elisenda Feliu, University of Copenhagen \hfill{\tt efeliu@math.ku.dk}\\
Paul Alexander Helminck, IMPAN Warsaw \hfill{\tt phelminck@impan.pl}\\
Oskar Henriksson, CSBD and MPI-CBG Dresden \hfill{\tt oskar.henriksson@mpi-cbg.de}\\
Yue Ren, Durham University \hfill{\tt yue.ren2@durham.ac.uk}\\
Benjamin Schröter, KTH Stockholm \hfill {\tt schrot@kth.se}\\
Máté L. Telek, Budapest University of Technology and Economics \hfill {\tt mtelek@math.bme.hu}

 }

\end{document}